\documentclass[reqno,11pt]{amsart}
\usepackage{amssymb}
\usepackage{amsmath,mathrsfs}
\usepackage{mathtools}
\usepackage{mdwtab}
\usepackage{longtable}
\usepackage{subcaption}
\usepackage{stmaryrd}
\usepackage{tikz}
\usetikzlibrary{positioning,fit}
\usepackage[indent]{parskip}
\usepackage{url}
\usepackage{enumitem}
\usepackage{fullpage}
\usepackage{pstricks,pst-node}
\usepackage{makecell}
\usepackage{multirow}

\usepackage{bm}
\usepackage{threeparttable}
\usepackage{pifont}
\usepackage{amssymb}
\usepackage{caption}

\usepackage{graphicx}
\usepackage{booktabs}
\usepackage{comment, amsgen, amscd, xspace, epsfig, float, epic}
\usepackage{geometry}
\usepackage{multirow}
\usepackage{tikz}
\usepackage{threeparttable}
\usepackage{enumitem}
\usepackage{fancyhdr}
\usepackage[mark=o,edge length=1.2cm,root-radius=0.1cm]{dynkin-diagrams}
\usepackage[colorlinks,linkcolor=blue,urlcolor=cyan,citecolor=blue]{hyperref}

\theoremstyle{definition}

\newtheorem{Exam}[equation]{Example}

\newtheorem{Def}[equation]{Definition}
\newtheorem{Thm}[equation]{Theorem}

\newtheorem{Lem}[equation]{Lemma}

\newtheorem{Prop}[equation]{Proposition}

\newcommand{\hobox}[3]{\draw (0+#1,0-#2) rectangle (1+#1,-1-#2)++(-0.5,+0.5) node {$ #3$};}

\newcommand{\domscale}{0.5}

\numberwithin{equation}{section}

\newcommand{\fg}{{\mathfrak g}}
\newcommand{\fh}{{\mathfrak h}}

\newcommand{\fl}{{\mathfrak l}}

\newcommand{\fq}{{\mathfrak q}}

\newcommand{\fu}{{\mathfrak u}}

\allowdisplaybreaks[1]

\begin{document}
	
	\title[Gelfand--Kirillov dimension]{Reducibility of scalar generalized Verma modules of minimal parabolic type II}
	

	\author{Jing Jiang* and Siying Wu}
	\address{
		School of mathematical Sciences, East China Normal University, Shanghai 200241, China} 
	\email{jjsd6514@163.com (Jiang), hnrfc980@163.com (Wu)}
	
	
	\begin{abstract}
		Let $\mathfrak{g}$ be a exceptional complex simple Lie algebra and $\mathfrak{q}$ be a parabolic subalgebra. A generalized Verma module $M$ is called a scalar generalized Verma module if it is induced from a one-dimensional representation of $\mathfrak{q}$. In this paper, we will determine the first diagonal-reducible point of scalar generalized Verma modules associated to minimal parabolic subalgebras of exceptional Lie algebras by computing explicitly the Gelfand--Kirillov (GK) dimension of the corresponding highest weight modules.
		
	\end{abstract}
	\footnote{*Corresponding author.}
	\subjclass[2020]{16S30, 17B10, 17B20, 22E47}
	
	\keywords{Gelfand--Kirillov dimension; Minimal parabolic subalgebra; Exceptional type; PyCox}
	
	\maketitle
	
	
	\section{Introduction}
	
	Let $\mathfrak{g}$ be a finite-dimensional complex simple Lie
	algebra and $U(\mathfrak{g})$ be its universal enveloping algebra.
	Fix a Cartan subalgebra $\mathfrak{h}$ and denote by $\Delta$ the root system associated to $(\mathfrak{g}, \mathfrak{h})$. Choose a positive root system
	$\Delta^+\subset\Delta$ and a simple system $\Pi\subset\Delta^+$. Denote $\rho$ the half sum of positive roots. Let $\mathfrak{g}=\bar{\mathfrak{n}}\oplus\mathfrak{h}\oplus\mathfrak{n}$
	be the triangular decomposition of $\mathfrak{g}$ with nilpotent radical $\mathfrak{n}$ and its opposite nilradical $\bar{\mathfrak{n}}$.   Choose a subset $I\subset\Pi$ , which generates a subsystem
	$\Delta_I\subset\Delta$.
	Let $\mathfrak{q}_I$ be the standard parabolic subalgebra corresponding to $I$ with Levi decomposition $\mathfrak{q}_I=\mathfrak{l}_I\oplus\mathfrak{u}_I$. So when $I=\emptyset$, we have $\mathfrak{q}_{\emptyset}=\mathfrak{h}\oplus
	\mathfrak{n}=\mathfrak{b}$.
	
	Let $\mathfrak{q}_I=\mathfrak{l}_I\oplus\mathfrak{u}_I$ and $F_{\lambda}$ be a finite-dimensional irreducible $\mathfrak{l}_I$-module with the highest weight $\lambda\in\mathfrak{h}^*$. It can also be viewed as a
	$\mathfrak{q}_I$-module with trivial $\mathfrak{u}_I$-action. The {\it generalized Verma module} $M_I(\lambda)$ is defined by
	\[
	M_I(\lambda):=U(\mathfrak{g})\otimes_{U(\mathfrak{q})}F_{\lambda}.
	\]
	The irreducible quotient of $M_I(\lambda)$ is denoted by $L(\lambda)$, which is a highest weight module with highest weight $\lambda$.  It is also the irreducible quotient of $M_{\emptyset}(\lambda)$. In the case when $\dim(F_{\lambda})=1$, we usually denote it by $\mathbb{C}_{\lambda}$ and $M_I(\lambda)$ is called a {\it scalar generalized Verma module}.
	
	The theory of highest weight modules of simple complex finite-dimensional Lie algebras is built on the original work of Verma \cite{VD}. In that work, Verma introduced and studied the generic family of highest weight modules known as Verma modules, which have become the foundation of the field. There have been several attempts to extend the theory of Verma modules, and one of the most natural ways is to generalize the Verma modules themselves.  This can be accomplished in various ways, such as in \cite{AG1,AR,JL}. Generalized Verma modules (GVM) have been studied from various perspectives, and many of the properties of classical Verma modules have either been proved in GVMs or have been generalized. For a study of GVM, see Mazorchuk's work in \cite{MV}. In this article, he provided a comprehensive survey of GVM for parabolic subalgebras of simple Lie algebras using parabolic induction and outlined the historical development of the subject.

	The problem of reducibility of generalized Verma modules is of great significance in representation theory and is closely related to several other problems, as described in \cite{BXiao, EHW, MH}. The essential tool for solving this problem is Jantzen's criterion \cite{JC}, although it can be quite complicated in practical applications. However, Kubo \cite{Ku} found some practical reducibility criteria for solving the reducibility of scalar generalized Verma modules associated with exceptional simple Lie algebras and certain maximal parabolic subalgebras. Using Kubo's results, He \cite{HH} established the reducibility of all scalar generalized Verma modules of Hermitian symmetric pairs. Then He--Kubo--Zierau \cite{HKZ} extended it to all scalar generalized Verma modules associated with maximal parabolic subalgebras. Recently, Bai--Xiao \cite{BXiao} solved the reducibility problem for all generalized Verma modules of Hermitian symmetric pairs.
	
	The GK dimension plays a fundamental role in describing algebraic structures with infinite dimensions. It has been used to measure the size of the representations of Lie algebras and Lie group  since Joseph's work \cite{AJ1}. Recently, a study led by Bai--Xiao \cite{BX} proved that a scalar generalized Verma module is reducible if the GK dimension of its simple quotient is smaller than the dimension of $\fu$. (This approach does not rely on the reduction method outlined in \cite{HKZ}). For classical Lie algebras, by using this result and the algorithm of computing GK
	dimension as in \cite{BX1,BXX}, Bai--Jiang \cite{BJ} gave a new proof for the complete list of highest weights for which the
	scalar generalized Verma modules associated to maximal parabolic subalgebras are reducible. Jiang \cite{JJ} gave a proof for the complete list of highest weights for which the
	scalar generalized Verma modules associated to classical minimal parabolic subalgebras are reducible. For exceptional Lie algebras, we can use the same technique as in \cite{BGXW} to compute the GK dimension of the scalar generalized Verma modules associated to minimal parabolic subalgebras. The following web page in \cite{BGXW} is convenient to compute the GK dimension of $L(\lambda)$:
	\begin{center}
		\textcolor{blue}{http://test.slashblade.top:5000/lie/GKdim},
	\end{center}
	this program takes $\lambda$ as the input and returns
$\operatorname{GKdim} L(\lambda)$.

	The paper is organized as follows. The necessary preliminaries for GK dimension,  minimal parabolic subalgebra, Robinson--Schensted (RS) insertion algorithm, Lusztig's {\bf a}-function and PyCox are given in  \S \ref{pre}.  The reducibility of scalar generalized Verma modules of type $G_2$ and $F_4$ are given in \S \ref{type_gf}. In \S \ref{type_e}, we give the reducibility of scalar generalized Verma modules of type $E$. In \S \ref{app}, we tabulate the sets of reducible points of $M_{I}(z\widehat{\xi_p})$ corresponding to each exceptional Lie algebra.
	
	 \section*{Acknowledgments}
	  The authors thank Li Luo and the anonymous referees for providing many constructive comments and helping in improving the contents of our paper. The first author is grateful for the financial support provided by the Program of China Scholarship Council (Grant No. 202506140041).
	
	\section{Preliminaries}\label{pre}
	In this section, we will give brief preliminaries for GK dimension, minimal parabolic subalgebra, Lusztig's {\bf a}-function, RS algorithm, integral root system and PyCox. See \cite{BGXW,JJ} for more details.
	
	\subsection{Minimal parabolic subalgebra}
	A parabolic subalgebra $\mathfrak{q}_I$ is said to be minimal if $I$ is a minimal nonempty subset of $\Pi$. In other words, there is only one element in $I$.
	
	\begin{Lem}[{\cite[Lem. 3.2]{JJ}}]
		If $\fq_I$ is the minimal parabolic subalgebra, then $\dim(\fu)=\arrowvert\Delta^+\arrowvert-1$.
	\end{Lem}
	\begin{proof}
		We know $\fq_I=\fl_I\oplus\fu_I$ and $\dim(\Delta^+(\fl))=1$, then the value of $\dim(\fu)$ is obvious.
	\end{proof}
	Thus we have the following table.
	\begin{table}[htbp]
		\centering
		\renewcommand{\arraystretch}{1.4}
		\setlength\tabcolsep{5pt}
		\caption{$\dim(\fu)$ of $G$ with minimal parabolic}
		\label{constants}
		{	\begin{tabular}{llll}		
				\toprule
				$G$ & $|\Delta^+|$&   & $\dim(\fu)$ \\ \hline 
				$SL(n,\mathbb{C})$ & $\frac{n(n-1)}{2}$&  & $\frac{n(n-1)}{2}-1$  \\
				$SO(2n+1,\mathbb{C})$  & $n^2$&  & $n^2-1$ \\ 
				$Sp(2n,\mathbb{C})$  & $n^2$& & $n^2-1$   \\  
				$SO(2n,\mathbb{C})$ & $n^2-n$&  & $n^2-n-1$ \\ 
				$E_{6}$   & 36&  & $35$ \\  
				$E_{7}$  & 63& & $62$ \\
				$E_{8}$  & 120& & $119$ \\
				$F_{4}$  & 24& & $23$ \\
				$G_{2}$  & 6& & $5$ \\
				\bottomrule
				
		\end{tabular}}
	\end{table}

	\subsection{Lusztig's {\bf a}-function}
	Let $(W, S)$ be a Coxeter group with Coxeter matrix $\left(m_{s t}\right)_{s, t \in S}$ and length function $\ell$. Following the seminal papers \cite{KL,lusztig1985cellsI}, we define the Hecke algebra $\bm{\mathcal{H}}$ of $(W, S)$ as follows. Let $q^{\frac{1}{2}}$ be an indeterminate and $\mathbb{Z}\left[q^{ \pm \frac{1}{2}}\right]$ be the ring of Laurent polynomials in $q^{\frac{1}{2}}$. Let $\bm{\mathcal{H}}$ be a free $\mathbb{Z}\left[q^{ \pm \frac{1}{2}}\right]$-module with a formal basis $\left\{\widetilde{T}_w\right\}_{w \in W}$. The multiplication on $\bm{\mathcal{H}}$ is defined by
	
	$$
	\widetilde{T}_s \widetilde{T}_w:=\left\{\begin{array}{ll}
		\widetilde{T}_{s w}, & \text { if } \ell(s w)>\ell(w) \\
		\left(q^{\frac{1}{2}}-q^{-\frac{1}{2}}\right) \widetilde{T}_w+\widetilde{T}_{s w}, & \text { if } \ell(s w)<\ell(w)
	\end{array}, \quad \forall s \in S, w \in W.\right.
	$$
	Then $\bm{\mathcal{H}}$ is an associative $\mathbb{Z}\left[q^{ \pm \frac{1}{2}}\right]$-algebra with unity $\widetilde{T}_e$. Here $\widetilde{T}_w=q^{-\frac{\ell(w)}{2}} T_w$ where $T_w$ is the standard basis element in \cite{KL}.
	
	Besides the normalized standard basis $\left\{\widetilde{T}_w \mid w \in W\right\}$, we have another basis $\left\{C_w \mid w \in W\right\}$ in \cite{KL}(now called the Kazhdan-Lusztig basis), such that
	
	$$
	\begin{aligned}
		C_w & =\sum_{y \in W}(-1)^{\ell(w)+\ell(y)} q^{\frac{1}{2}(\ell(w)-\ell(y))} P_{y, w}\left(q^{-1}\right) \widetilde{T}_y \\
		& =\sum_{y \in W}(-1)^{\ell(w)+\ell(y)} q^{-\frac{1}{2}(\ell(w)-\ell(y))} P_{y, w}(q) \widetilde{T}_{y^{-1}}^{-1},
	\end{aligned}
	$$
	where $P_{y, w} \in \mathbb{Z}[q], P_{y, w}=0$ unless $y \leq w, P_{w, w}=1$, and $\operatorname{deg}_q P_{y, w} \leq \frac{1}{2}(\ell(w)-\ell(y)-1)$ if $y<w$. Here $\leq$ is the Bruhat order on $W$, and $y<w$ indicates $y \leq w$ and $y \neq w$.
	
	For any $x, y, w \in W$, let $h_{x, y, w} \in \mathbb{Z}\left[q^{ \pm \frac{1}{2}}\right]$ be such that $C_x C_y=\sum_{w \in W} h_{x, y, w} C_w$. Following G. Lusztig \cite{lusztig1985cellsI}, for each $w \in W$ we define $${\bf a}(w):=\min \left\{i \in \mathbb{N} \left\lvert\, q^{\frac{i}{2}} h_{x, y, w} \in \mathbb{Z}\left[q^{\frac{1}{2}}\right]\right., \forall x, y \in W\right\}.$$ Then ${\bf a}(w)$ is a well defined natural number. We say that the function ${\bf a}$ is bounded if there is $N \in \mathbb{N}$ such that ${\bf a}(w) \leq N$ for any $w \in W$.
	
	For $x, y \in W$, we say $y \underset{\text { LR }}{\leqslant} x$ if there exist $H_1, H_2 \in \bm{\mathcal{H}}$, such that $C_y$ has nonzero coefficient in the expression of $H_1 C_x H_2$ with respect to the basis $\left\{C_w\right\}_{w \in W}$. We say $x \underset{\mathrm{LR}}{\sim} y$ if $x \underset{\mathrm{LR}}{\leqslant} y$ and $y \underset{\mathrm{LR}}{\leqslant} x$. It turns out that $\underset{\mathrm{LR}}{\leqslant}$ is a pre-order on $W$, and $\underset{\mathrm{LR}}{\sim}$ is an equivalence relation on $W$. The equivalence classes are called two-sided cells. The set of two-sided cells forms a partial order set with respect to $\underset{\text{LR}}{\leqslant}$. 
	
	We have the following standard properties of Kazhdan-Lusztig basis $\left\{C_w\right\}_{w \in W}$, two-sided cells and Lusztig's ${\bf a}$-function.
	
	\begin{Prop}[{\cite{KL,lusztig1985cellsI,lusztig2003hecke}}]\label{a_function}
		Let $s \in S$ and $w, x, y \in W$.
		
		\begin{enumerate}
			\item  $P_{y, w}(0)=1$ for any $y \leq w$.
			
			\item ${\bf a}(w)=0$ if and only if $w=e$. We have ${\bf a}(s)=1$ for any $s \in S$.
			
			\item ${\bf a}(w)={\bf a}(w^{-1})$.

			\item If $m_{r t}<\infty$ for some $r, t \in S$ with $r \neq t$, then ${\bf a}\left(w_{r t}\right)=m_{r t}$.
			
			\item If $y \underset{\mathrm{LR}}{\leqslant} x$, then ${\bf a}(y) \geq {\bf a}(x)$. If $y \underset{\mathrm{LR}}{\sim} x$, then ${\bf a}(y)={\bf a}(x)$.
			
			\item Suppose ${\bf a}$ is bounded on $W$. If ${\bf a}(y)={\bf a}(x)$ and $y \underset{\mathrm{LR}}{\leqslant} x$, then $y \underset{\mathrm{LR}}{\sim} x$.
			
			\item If $W$ is a direct product of Coxeter subgroups $W_1$ and $W_2$, then
			$${\bf a}(w)={\bf a}(w_1)+{\bf a}(w_2)$$
			for $w=(w_1,w_2)\in W_1\times W_2=W.$
		\end{enumerate}   
	\end{Prop}
	
	\subsection{PyCox}  In this paper, for a simple root $\alpha_i\in\Pi$ we simply write the simple reflection $s_{\alpha_i}$ by $s_i$. Here we adopt Geck's notations in \cite{Geck}, which can be used to compute the value of ${\bf a}(w)$ for $w\in W$.  In type $F_4$ and $E_n$, we use $[i_1-1,i_2-1,\cdots,i_k-1]$ to represent $w=s_{i_1}s_{i_2}\cdots s_{i_k}$. In PyCox the function ‘klcellrepem’ will give us the ${\bf a}$ value of $w\in W$ and the character of the left cell to which $w$ belongs (note that the notation for characters in PyCox is different with Carter \cite{Ca85}).
	
	\begin{Exam}
		In the case of $E_7$, we use $[3,1,3]$ to denote the element $w=s_4s_2s_4$. Then by PyCox we have:
		
		\texttt{>>> W = coxeter("E",7)}
		
		\texttt{>>>print(klcellrepem(W,[3,1,3]))}
		
		\texttt{\{'size': 77, 'character': [['56\_a',1],['21\_a',1]], 'a': 3,}
		
		\texttt{'special': '56\_a', 'index': 8, 'elms': False, 'distinv': False\}. }
		
		Then we have ${\bf a}(w)=3$, the corresponding characters are $56_a$ and $21_a$,  and we can easily compute that the Gelfand--Kirillov dimension is 60 since GKdim $L(\lambda)=|\Delta^+|-{\bf a}(w)$ \cite{BX1}.
		
	\end{Exam}
	
	\subsection{RS algorithm}\label{R-S} We recall the RS algorithm which will be used in our paper. Some details can be found in \cite{BX1}.
	
	For  a totally ordered set $ \Gamma $, we  denote by $ \mathrm{Seq}_n (\Gamma)$ the set of sequences $ x=(x_1,x_2,\dots, x_n) $   of length $ n $ with $ x_i\in\Gamma $. In our paper, we usually take $\Gamma$ to be $\mathbb{Z}$ or a coset of $\mathbb{Z}$ in $\mathbb{C}$.
	Then we have a  Young tableau $P(x)$ obtained by applying the following RS algorithm  to $x\in \mathrm{Seq}_n (\Gamma)$. 
	\begin{Def}[RS algorithm]
		For an element  $ x \in  \mathrm{Seq}_n (\Gamma)$, we write  $x=(x_1,\dots,x_n)$. We associate to $x $ a  Young tableau  $ P(x) $ as follows: Let $ P_0 $ be an empty Young tableau. Assume that we have constructed the Young tableau $ P_k $ associated to $ (x_1,\dots,x_k) $, $ 0\leq k<n $. Then $ P_{k+1} $ is obtained by adding $ x_{k+1} $ to $ P_k $ as follows. Firstly we add $ x_{k+1} $ to the first row of $ P_k $ by replacing the leftmost entry $ x_j $ in the first row which is \textit{strictly} bigger than $ x_{k+1} $.  (If there is no such an entry $ x_j $, we just add a box with entry $x_{k+1}  $ to the right side of the first row, and end this process). Then add this $ x_j $ to the next row as the same way of adding $x_{k+1} $ to the first row.  Finally we put $P(x)=P_n$.
		
	\end{Def}

	Denote the shape of the Young tableau $P(x)$ by ${\bf p}(x)=\mathrm{sh}(P(x))=[p_1,p_2,...,p_k]$, where $p_i$ is the number of boxes in the $i$-th
	row of $P(x)$. Then ${\bf p}(x)$ is a partition of $n=\sum\limits_{1\leq i\leq k} p_i$.
	
	\begin{Exam}

		Suppose $x=(1,-3,2,-1,-3)$.  Usually  we write $$x=(1,-3,2,-1,-3')$$ and regard $-3<-3'$. Then from the RS algorithm, we have
		$$
		\begin{tikzpicture}[scale=\domscale+0.1,baseline=-10pt]
			\hobox{0}{0}{1}
			
		\end{tikzpicture}\stackrel{-3}{\to}
		\begin{tikzpicture}[scale=\domscale+0.1,baseline=-18pt]
			\hobox{0}{0}{-3}
			\hobox{0}{1}{1}
			
		\end{tikzpicture}\stackrel{2}{\to}  
		\begin{tikzpicture}[scale=\domscale+0.1,baseline=-18pt]
			\hobox{0}{0}{-3}
			\hobox{1}{0}{2}
			\hobox{0}{1}{1}
			
		\end{tikzpicture}\stackrel{-1}{\to}
		\begin{tikzpicture}[scale=\domscale+0.1,baseline=-18pt]
			\hobox{0}{0}{-3}
			\hobox{1}{0}{-1}
			\hobox{1}{1}{2}
			\hobox{0}{1}{1}
			
		\end{tikzpicture}\stackrel{-3}{\to}
		\begin{tikzpicture}[scale=\domscale+0.1,baseline=-18pt]
			\hobox{0}{0}{-3}
			\hobox{1}{0}{-3}
			\hobox{0}{2}{1}
			\hobox{0}{1}{-1}
			\hobox{1}{1}{2}
		\end{tikzpicture}=P(x).$$
		Thus we have ${\bf p}(x)=[2,2,1]$, which is a partition of $5$.
				
				

	\end{Exam}

	\subsection{Gelfand--Kirillov dimension}
	
	Let $M$ be a finite generated $U(\mathfrak{g})$-module. Fix a finite-dimensional generating subspace $M_0$ of $M$. Let $U_{n}(\mathfrak{g})$ be the standard filtration of $U(\mathfrak{g})$. Set $M_n=U_n(\mathfrak{g})\cdot M_0$ and
	\(
	\text{gr} (M)=\bigoplus\limits_{n=0}^{\infty} \text{gr}_n M,
	\)
	where $\text{gr}_n M=M_n/{M_{n-1}}$. Thus $\text{gr}(M)$ is a graded module of $\text{gr}(U(\mathfrak{g}))\simeq S(\mathfrak{g})$.

	\begin{Def} The \textit{Gelfand--Kirillov dimension} of $M$  is defined by
		\begin{equation*}
			\operatorname{GKdim} M = \varlimsup\limits_{n\rightarrow \infty}\frac{\log\dim( U_n(\mathfrak{g})M_{0} )}{\log n}.
		\end{equation*}
	\end{Def}
	
	It is easy to see that the above definition is independent of the choice of $M_0$. Then we have the following lemma.
	\begin{Lem}[{\cite[Lem. 4.4]{BX}}]\label{GKdown}
		For any $z\in\mathbb{C}$, we have
		\[
		\operatorname{GKdim}(L((z+1)\xi))\leq\operatorname{GKdim}(L(z\xi)),
		\]
		where $\xi$ is the fundamental weight of a simple root $\alpha$.
		In particular, if $M_I(z\xi)$ is reducible, then $M_I((z+1)\xi)$ is also reducible.
	\end{Lem}
	
	The following lemma is very useful in our proof of the main results.
	\begin{Lem}[{\cite[Thm. 1.1]{BX}}]\label{reducible}
		A scalar generalized Verma module $M_I(\lambda)$ is irreducible if and only if ${\rm GKdim}\:L(\lambda)=\dim(\mathfrak{u})$.
	\end{Lem}
	
	There is a formula connecting the Lusztig's {\bf a}-function and the GK dimension
	$$\text{GK}\dim L(\lambda)=|\Delta^+|-{\bf a}(w_{\lambda})\text{~for~some~}w_{\lambda}\in W$$
	in \cite{BX1}, it also holds for exceptional types. In \cite{BGXW} the authors have found an efficient algorithm to compute the Gelfand--Kirillov dimension of highest weight modules of exceptional Lie algebras and we recall their notations and algorithms here.
	
	Let $(-, -): \mathfrak{h} \times \mathfrak{h}^* \to \mathbb{C}$ be the canonical pairing. For $\mu\in\fh^*$, define
	$$\Delta_{[\mu]}:=\{\alpha\in\Delta~|~(\mu,\alpha^{\vee})\in\mathbb{Z}\},$$
	where $\alpha^{\vee}$ is the coroot of the root $\alpha\in\Delta$. Denote
	$$W_{[\mu]}:=\{w\in W~|~w\mu-\mu\in\mathbb{Z}\Delta\}.$$
	Then $\Delta_{[\mu]}$ is a root system and let $\Pi_{[\mu]}$ be the simple system of $\Delta_{[\mu]}$. Set $S:=\{\alpha\in\Pi_{[\mu]}~|~(\mu,\alpha^{\vee})=0\}$. Let $\ell$ be the length function on $W_{[\mu]}$ and 
	$$W_{[\mu]}^S:=\{w\in W_{[\mu]} ~|~\ell(ws_{\alpha})=\ell(w)+1\text{~for~all~}\alpha\in S \}.$$
	
	A weight $\mu$ is {\it anti-dominant} if $(\mu,\alpha^{\vee})\notin\mathbb{Z}_{>0}$ for all $\alpha\in\Delta^+$. For any $\lambda\in\fh^*$, there exists a unique anti-dominant weight $\mu\in\fh^*$ and a unique $w_{\lambda}\in W_{[\mu]}^S$.
	
	For integral highest weight modules, we can easily find the required $w_{\lambda}$ by using the positive index reduction algorithm \cite[Lem. 3.2]{BGXW}, then by using PyCox we can  compute the GK dimension. For non-integral highest weight modules, we need more notations.
	
	For a root system $\Delta$, denote $\mathfrak{h}_{\Delta}$ the corresponding Cartan subalgebra of Lie algebra $\mathfrak{g}_{\Delta}$ associated to $\Delta$. Denote $\Delta_{[\lambda]}$ the integral root system for $\lambda\in\mathfrak{h}^*_{\Delta}$ and $\Pi_{[\lambda]}$ the simple system of $\Delta_{[\lambda]}\cap\Delta^+$ in $\Delta_{[\lambda ]}$. It is provided in \cite{BGXW} a way to find $\Pi_{[\lambda]}$ as following. 
	\begin{Prop}[{\cite[Lem. 3.5]{BGXW}}]\label{integral_root_system}
		For $\lambda \in \mathfrak{h}_{\Delta}^*$, decompose $\Delta_{[\lambda]}$ into several orthogonal components:
		$$
		\Delta_{[\lambda]}=\Delta_{[\lambda]_1} \cup \Delta_{[\lambda]_2} \cup \cdots \cup \Delta_{[\lambda]_k} \simeq \Delta_1 \times \Delta_2 \times \cdots \times \Delta_k
		$$
		such that $\Delta_{[\lambda]_i} \simeq \Delta_i$ for $1 \leqslant i \leqslant k$. Then we have the following way to determine the simple system $\Pi_{[\lambda]_i}$ of $\Delta_{[\lambda]_i}$ :
		\begin{enumerate}
			\item  If $\left|\Delta_{[\lambda]_i}^{+}\right|=\frac{n(n+1)}{2}$, then $\Delta_{[\lambda]_i} \simeq A_n$. We compute $\rho_i=\frac{1}{2} \sum\limits_{\alpha \in \Delta_{[\lambda]_i}^{+}} \alpha$. Then we find out all the positive roots $\alpha \in \Phi_{[\lambda]_i}^{+}$ such that $(\rho_i, \alpha^{\vee})=1$ and denote them by $I_n=\left\{\alpha_{i_1}, \alpha_{i_2}, \ldots, \alpha_{i_n}\right\}$. Then we have $\Pi_{[\lambda]_i}=I_n$.
			\item If $\left|\Delta_{[\lambda]_i}^{+}\right|=n^2$ and the number of short roots is $n\:(n \geqslant 2)$, then $\Delta_{[\lambda]_i} \simeq B_n$. We compute $\rho_i=\frac{1}{2} \sum\limits_{\alpha \in \Delta_{[\lambda]_i}^{+}} \alpha$. Then we find out all the positive roots $\alpha \in \Delta_{[\lambda]_i}^{+}$ such that $(\rho_i, \alpha^{\vee})=1$ and denote them by $I_n=\left\{\alpha_{i_1}, \alpha_{i_2}, \ldots, \alpha_{i_n}\right\}$. Then we have $\Pi_{[\lambda]_i}=I_n$.
			\item If $\left|\Delta_{[\lambda] i}^{+}\right|=n^2$ and the number of long roots is $n(n \geqslant 3)$, then $\Delta_{[\lambda]_i} \simeq C_n$. Then similar to the case in (2), we have $\Pi_{\left[\lambda]_i\right.}=I_n$.
			\item If $\left|\Delta_{[\lambda]_i}^{+}\right|=n^2-n(n \geqslant 4)$, then $\Delta_{[\lambda]_i} \simeq D_n$. Then similar to the case in (1), we have $\Pi_{[\lambda]_i}=I_n$.
			\item If $\left|\Delta_{[\lambda]_i}^{+}\right|=36$, then $\Delta_{[\lambda]_i} \simeq E_6$. Then similar to the case in (1), we have $\Pi_{[\lambda]_i}=I_6$.
			\item If $\left|\Delta_{[\lambda]_i}^{+}\right|=63$, then $\Delta_{[\lambda]_i} \simeq E_7$. Then similar to the case in (1), we have $\Pi_{[\lambda]_i}=I_7$.
			
		\end{enumerate} 
	\end{Prop}
	
	Denote the label ‘$\:\widetilde\:\:$’  attached to the connected component corresponding to short roots. In other cases it is similar. Let $\phi:\Delta_{[\lambda]}\to\phi(\Delta_{[\lambda]})$ be an isomorphism, then we have ${\bf a}(w)={\bf a}(\phi(w))$.  We will determine the value of ${\bf a}(w_\lambda)$ by Proposition \ref{integral_root_system} in the non-integral case. Then we have the following  proposition.

		\begin{Prop}{\cite[Thm. 4.1, 5.1, 6.1, 6.4 and 6.7]{BGXW}}]\label{integral_case}
			Let $L(\lambda)$ be a simple integral highest weight module of type $E$, $F$ and $G$, then the following holds:
			\begin{enumerate}
				\item ${\bf a}(w_\lambda)=|\Delta^+|$ if and only if  $\lambda+\rho$ is dominant integral.
				
				\item ${\bf a}(w_\lambda)=0$ if and only if $\lambda+\rho$ is anti-dominant integral.
				
			\end{enumerate}
		\end{Prop}
		
				
				
				
				

			
			If $\lambda\in\fh^*$ is non-integral, we can get an integral weight $\phi(\lambda)$ of type $\phi(\Delta_{[\lambda]})$ by Proposition \ref{integral_root_system}, in this case we have ${\bf a}(w_{\lambda})={\bf a}(w_{\phi(\lambda)})$ and then we can compute the value of ${\bf a}(w_{\phi(\lambda)})$ by RS algorithm and Proposition \ref{a_function} if $\phi(\Delta_{[\lambda]})$ is of classical type and by PyCox and Proposition \ref{a_function} if $\phi(\Delta_{[\lambda]})$ is of exceptional type.
			From now on, we will use $(k_1,k_2,\cdots,k_t)$ to denote  $\lambda=k_1e_1+k_2e_2\cdots+k_te_t$, where $\{e_1,e_2,\cdots,e_n\}$ is a basis for $\fh^*$.
			\begin{Exam}
				Let $\fg=F_4$ and $L(\lambda)$ be the highest weight module of $\fg$ with $\lambda+\rho=(1,2,\frac{3}{2},\frac{7}{2})$, then $\Delta_{[\lambda]}\simeq C_4$ by Proposition \ref{integral_root_system}. Suppose simple system of $C_4$ is $\Pi=\{\beta_1=e_1-e_2,\beta_2=e_2-e_3,\beta_3=e_3-e_4,\beta_4=2e_4\}$. Define a map $\phi :\Delta_{[\lambda]}\to C_4$ by
				$$\phi(e_2)=\beta_1, \phi(\alpha_1)=\beta_2,\phi(\alpha_2)=\beta_3,\phi(\alpha_3)=\beta_4.$$
				Then we have
				$$\lambda+\rho=-2\xi_1+7\xi_2-6\xi_3+4\xi_4,$$
				where $\{\xi_i|1\leq i\leq 4\}$ are the fundamental weights for $\Pi_{[\lambda]}$, and 
				$$\phi(\lambda+\rho)=-2\phi(\xi_1)+7\phi(\xi_2)-6\phi(\xi_3)+4\phi(\xi_4)=(3,-1,5,-2),$$
				which is an integral weight of type $C_4$, we can get that $a(w_{\lambda})=6$ by the RS algorithm. Thus $\operatorname{GKdim} L(\lambda)=|\Delta^+|-{\bf a}(w_{\lambda})=24-6=18$ by \cite{BX1}.
			\end{Exam}

			\section{Reducibility of scalar generalized Verma modules for type $G_2$ and $F_4$}\label{type_gf}
Let $\mathfrak{g}$ be a finite-dimensional complex semisimple Lie
algebra and let $\mathfrak{b}=\mathfrak{h}\oplus\mathop\oplus\limits_{\alpha\in\Delta^+}\fg_\alpha$
be a fixed Borel subalgebra of $\fg$. For a minimal parabolic subalgebra $\fq$, we recall that $\fq$ corresponds to the subsets  $\Pi\setminus\{\alpha_i\}_{i\ne p}$. It is easy to get that  $\lambda=z\eta$ for some $z\in\mathbb{R}$ and $\eta=\sum\limits_{i\neq p}k_i\xi_i$ by Weyl dimension formula, where $k_i\in\mathbb{R}$ and $\xi_i$ is the fundamental weight associated with the simple root $\alpha_i$. In this paper, we suppose that $\eta=\sum\limits_{i\neq p}\xi_i=\widehat{\xi_p}$ and we call the corresponding reducible point a {\it diagonal-reducible point}. 

From Lemma \ref{GKdown}, the set of diagonal-reducible points of  a scalar generalized Verma module $M_I(z\widehat{\xi_p})$ is given in the following diagram:
\vspace{1em}
\begin{center}
\begin{tikzpicture}
    \draw[thick] (-1,0) -- (3,0);
    \coordinate (a) at (3,0);
    \fill (a) circle (2pt) node[below] {$a$};
    \coordinate (c) at (4,0);
    \fill (c) circle (2pt);
     \draw (4,0) -- (4,-0.4);
     \coordinate (d) at (5,0);
    \fill (d) circle (2pt);
    \draw (5,0) -- (5,-0.4);
    \coordinate (b) at (4.5,0);
    \draw[->, >=Triangle] (4.7,-0.2) -- (5,-0.2);
    \draw[<-, >=Triangle] (4,-0.2) -- (4.3,-0.2);
    \node at (4.5,-0.2) {$b$};
    \node at (5.5,0) {$\ldots$};
    \node at (6,0) {$\ldots$};
\coordinate (e) at (6.5,0);
    \fill (e) circle (2pt);
   \node at (7,0) {$\ldots$};
    \node at (7.5,0) {$\ldots$};  
\end{tikzpicture}
\end{center}
where the diagonal-reducible points starting from $z=a\in \mathbb{R}$ are equally spaced at an interval of length $b$ and can be written in the form $a+b\mathbb{Z}_{\geq 0}$. The point $a$  will be called the {\it first diagonal-reducible point} of $M_I(z\widehat{\xi_p})$.

From  Lemma \ref{GKdown}, we only need to find the first diagonal-reducible point of the scalar generalized Verma module $M_{I}(z\widehat{\xi_p})$. The following lemma plays an important role in the process of determining the first diagonal-reducible point.

\begin{Lem}\label{fundamental_weight}
	Let $\xi_i$ be the fundamental weight of simple root $\alpha_i$ and $\rho$ be half the sum of the members in $\Delta^+$, then $\rho=\sum\limits_{i\geq 1}\xi_i$.
\end{Lem}
\begin{proof}
It can be found in Lemma A of \S 13.3 in \cite{Hum78}. Of course, we can also draw the conclusion case by case  in \cite{Knapp}. For convenience, from now on we will use Knapp conventions to denote the simple root in type $E$ $F$ and $G$, see \cite[Appendices C.2.]{Knapp} for more details.
\end{proof}

In this section we give the reducibility of scalar generalized Verma modules associated to minimal parabolic subalgebra for type $G_2$ and $F_4$. It's worth noting that in  $G_2$, the case of minimal parabolic type is contrary to maximal parabolic type since there are only two simple roots. In other words,  in the minimal parabolic type, the case of $p=1$ is the same as $p=2$  in the maximal parabolic type.

\begin{Thm}\label{reducible_g2}
	Let $\fg$ be of type $G_2$, then $M_I(z\widehat{\xi_p})$ is reducible if and only if
	\begin{enumerate}
		\item $p=1$, $z\in \{-\frac{1}{2}+\frac{1}{2}\mathbb{Z}_{\geq 0}\}\cup\{-\frac{2}{3}+\frac{1}{3}\mathbb{Z}_{\geq 0}\}$; or
		
		\item $p=2$, $z\in\{-\frac{3}{2}+\mathbb{Z}_{\geq 0} \}\cup\{\mathbb{Z}_{\geq 0}\}$.
	\end{enumerate}
\end{Thm}	
\begin{proof}
	
	{\bf Step 1.} Let $\Delta^+(\fl)=\{\alpha_1\}=\{(1,-2,0)\}$. When $M_I(\lambda)$ is of scalar type, we know that $\lambda=z\xi_2$ for some $z\in\mathbb{R}$, where $\xi_2$ is the fundamental weight of simple root $\alpha_2$ and $\xi_2=(-1,-1,2)$.
	
	In \cite{Knapp} we know that $\rho=(-1,-2,3)$ and $\lambda+\rho=(-z-1,-z-2,2z+3)$.
	\begin{enumerate}
		\item If $\lambda$ is integral, then we have the following.
		\begin{enumerate}
			\item When $z=0$, then $\lambda+\rho=(-1,-2,3)$, which is dominant integral, by Proposition \ref{integral_case} we can get that $\operatorname{GKdim} L(\lambda)=0<\dim(\fu)$.
			
			By Lemma \ref{reducible} we can obtain that $z=0$ is a diagonal-reducible point. 
			
			\item When $z=-1$, then $\lambda+\rho=(0,-1,1)$, one can easily check that $\lambda+\rho$ is not dominant or anti-dominant, by Proposition \ref{integral_case} we can get that $\operatorname{GKdim} L(\lambda)=5=\dim(\fu)$.
			
			By Lemma \ref{reducible} we can obtain that $z=-1$ is an irreducible point. 
		\end{enumerate}
		\item If $\lambda$ is half-integral, one can easily check that $\Delta_{[\lambda]}=\Delta_1\bigcup\Delta_2\simeq  A_1\times \widetilde{A_1}$. The simple system of $\Delta_1$ is $\Pi_1=\{\beta_1=(-1,-1,2)\}$, and the simple system of $\Delta_2$ is $\Pi_2=\{\gamma_1=(1,-1,0)\}$. Suppose the simple system of $A_1$ is $\{\alpha_1=(1,-1,0)\}$, and the simple system of $\widetilde{A_1} $ is $\{\alpha_1^\prime=(1,-1,0)\}$. We define a map $\phi:\Delta_{[\lambda]}\to A_1\times \widetilde{A_1}$ such that $\phi(\beta_1)=\alpha_1,\phi(\gamma_1)=\alpha_1^\prime$, and we can have $\phi(\lambda+\rho)\mid_{A_1}=(\frac{1}{4}(2z+3),-\frac{1}{4}(2z+3))$, $\phi(\lambda+\rho)\mid_{\widetilde{A_1}}=(\frac{1}{2},-\frac{1}{2})$. We will have the following.
		
		\begin{enumerate}
			\item When $z=-\frac{1}{2}$, then $\lambda+\rho=(-\frac{1}{2},-\frac{3}{2},2)$. $\phi(\lambda+\rho)\mid_{A_1}=(\frac{1}{2},-\frac{1}{2})$, $\phi(\lambda+\rho)\mid_{\widetilde{A_1}}=(\frac{1}{2},-\frac{1}{2})$. By using RS algorithm, we have $a(w_{\lambda})=2$ and $\operatorname{GKdim} L(\lambda)=4<\dim(\fu)$.
			
			By Lemma \ref{reducible} we can obtain that $z=-\frac{1}{2}$ is a diagonal-reducible point. 
			
			\item  When $z=-\frac{3}{2}$, then $\lambda+\rho=(\frac{1}{2},-\frac{1}{2},0)$. $\phi(\lambda+\rho)\mid_{A_1}=(0,0)$, $\phi(\lambda+\rho)\mid_{\widetilde{A_1}}=(\frac{1}{2},-\frac{1}{2})$. By using RS algorithm, we have $a(w_{\lambda})=1$ and $\operatorname{GKdim} L(\lambda)=5=\dim(\fu)$.
			
			By Lemma \ref{reducible} we can obtain that $z=-\frac{3}{2}$ is an irreducible point. 
			
		\end{enumerate}
		\item If $\lambda$ is one-third-integral, one can easily check that $\Delta_{[\lambda]}\simeq  A_2$. The simple system of $\Delta_{[\lambda]}$ is $\Pi=\{\beta_1=(1,-1,0),\beta_2=(-1,0,1)\}$. Suppose the simple system of $A_2$ is $\{\alpha_1=(1,-1,0),\alpha_2=(0,1,-1)\}$. We define a map $\phi:\Delta_{[\lambda]}\to A_2$ such that $\phi(\beta_1)=\alpha_1,\phi(\beta_2)=\alpha_2$, and we can have $\phi(\lambda+\rho)=(z+2,z+1,-2z-3)$. We will have the following.
		
		\begin{enumerate}
			\item When $z=-\frac{2}{3}$, then $\lambda+\rho=(-\frac{1}{3},-\frac{4}{3},\frac{5}{3})$. $\phi(\lambda+\rho)=(\frac{4}{3},\frac{1}{3},-\frac{5}{3})$. By using RS algorithm, we have $a(w_{\lambda})=3$ and $\operatorname{GKdim} L(\lambda)=3<\dim(\fu)$.
			
			By Lemma \ref{reducible} we can obtain that $z=-\frac{2}{3}$ is a diagonal-reducible point. 
			
			\item When $z=-\frac{5}{3}$, then $\lambda+\rho=(\frac{2}{3},-\frac{1}{3},-\frac{1}{3})$. $\phi(\lambda+\rho)=(\frac{1}{3},-\frac{2}{3},\frac{1}{3})$. By using RS algorithm, we have $a(w_{\lambda})=1$ and $\operatorname{GKdim} L(\lambda)=5=\dim(\fu)$.
			
			By Lemma \ref{reducible} we can obtain that $-\frac{5}{3}$ is an irreducible point. 
			
			\item  When $z=-\frac{1}{3}$, then $\lambda+\rho=(-\frac{2}{3},-\frac{5}{3},\frac{7}{3})$. $\phi(\lambda+\rho)=(\frac{5}{3},\frac{2}{3},-\frac{7}{3})$. By using RS algorithm, we have $a(w_{\lambda})=3$ and $\operatorname{GKdim} L(\lambda)=3<\dim(\fu)$.
			
			By Lemma \ref{reducible} we can obtain that $z=-\frac{1}{3}$ is a diagonal-reducible point. 
			
			\item When $z=-\frac{4}{3}$, then $\lambda+\rho=(\frac{1}{3},-\frac{2}{3},\frac{1}{3})$. $\phi(\lambda+\rho)=(\frac{2}{3},-\frac{1}{3},-\frac{1}{3})$. By using RS algorithm, we have
			$a(w_{\lambda})=1$ and $\operatorname{GKdim} L(\lambda)=5=\dim(\fu)$.
			
			By Lemma \ref{reducible} we can obtain that $z=-\frac{4}{3}$ is an irreducible point. 
			
		\end{enumerate}
		
		\item If $\lambda$ is non-integral, non-half-integral and non-one-third-integral, one can easily check that $\Delta_{[\lambda]}\simeq  A_1$. The simple system of $\Delta_{[\lambda]}$ is $\Pi=\{\beta_1=e_1-e_2\}$. We can get $a(w_{\lambda})=1$ by using RS algorithm, thus $\operatorname{GKdim} L(\lambda)=5=\dim(\fu)$. In this case $M_I(z\widehat{\xi_1})$ is irreducible.
		
		All in all, when $p=1$, $ M_I(z\widehat{\xi_1})$ is reducible if and only if
		\begin{align*}
			z&\in \{-\frac{1}{2}+\mathbb{Z}_{\geq 0}\}\cup \{\mathbb{Z}_{\geq 0}\}\cup\{-\frac{2}{3}+\mathbb{Z}_{\geq 0}\}\cup\{-\frac{1}{3}+\mathbb{Z}_{\geq 0}\}\\&=\{-\frac{1}{2}+\frac{1}{2}\mathbb{Z}_{\geq 0}\}\cup\{-\frac{2}{3}+\frac{1}{3}\mathbb{Z}_{\geq 0}\}.
		\end{align*}
	\end{enumerate}
	
	{\bf Step 2.} Next we consider the case when $p=2$. Let $\Delta^+(\fl)=\{\alpha_2\}=\{(-2,1,1)\}$. When $M_I(z\widehat{\xi_p})$ is of scalar type, we know that $\lambda=z\xi_1$ for some $z\in\mathbb{R}$, where $\xi_1$ is the fundamental weight of simple root $\alpha_1$ and $\xi_1=(0,-1,1)$. 
	In \cite{Knapp} we know that $\rho=(-1,-2,3)$ and $\lambda+\rho=(-1,-z-2,z+3)$.
	\begin{enumerate}
		\item If $\lambda$ is integral, then we have the following.
		\begin{enumerate}
			\item When $z=0$, then $\lambda+\rho=(-1,-2,3)$, which is dominant integral, by Proposition \ref{integral_case} we can get that $\operatorname{GKdim} L(\lambda)=0<\dim(\fu)$.
			
			By Lemma \ref{reducible} we can obtain that $z=0$ is a diagonal-reducible point. 
			
			\item When $z=-1$, then $\lambda+\rho=(-1,-1,2)$, one can easily check that $\lambda+\rho$ is not dominant or anti-dominant, by Proposition \ref{integral_case} we can get that $\operatorname{GKdim} L(\lambda)=5=\dim(\fu)$.
			
			By Lemma \ref{reducible} we can obtain that $z=-1$ is an irreducible point. 
		\end{enumerate}
		\item If $\lambda$ is half-integral, one can easily check that $\Delta_{[\lambda]}=\Delta_1\bigcup\Delta_2\simeq  A_1\times \widetilde{A_1}$. The simple system of $\Delta_1$ is $\Pi_1=\{\beta_1=(-2,1,1)\}$, and the simple system of $\Delta_2$ is $\Pi_2=\{\gamma_1=(0,-1,1)\}$. Suppose the simple system of $A_1$ is $\{\alpha_1=(1,-1,0)\}$, and the simple system of $\widetilde{A_1}$ is $\{\alpha_1^\prime=(1,-1,0)\}$. We define a map $\phi:\Delta_{[\lambda]}\to A_1\times\widetilde{A_1}$ such that $\phi(\beta_1)=\alpha_1,\phi(\gamma_1)=\alpha_1^\prime$, and we can have $\phi(\lambda+\rho)\mid_{A_1}=(\frac{1}{2},-\frac{1}{2})$, $\phi(\lambda+\rho)\mid_{\widetilde{A_1}}=(z+\frac{5}{2},-z-\frac{5}{2})$.then we have the following.
		\begin{enumerate}
			\item When $z=-\frac{3}{2}$, then $\lambda+\rho=(-1,-\frac{1}{2},\frac{3}{2})$. $\phi(\lambda+\rho)\mid_{A_1}=(\frac{1}{2},-\frac{1}{2})$, $\phi(\lambda+\rho)\mid_{\widetilde{A_1}}=(1,-1)$. By using RS algorithm, we have $a(w_{\lambda})=2$ and $\operatorname{GKdim} L(\lambda)=4<\dim(\fu)$.
			
			By Lemma \ref{reducible} we can obtain that $z=-\frac{3}{2}$ is a diagonal-reducible point. 
			
			\item  When $z=-\frac{5}{2}$, then $\lambda+\rho=(-1,\frac{1}{2},\frac{1}{2})$. $\phi(\lambda+\rho)\mid_{A_1}=(\frac{1}{2},-\frac{1}{2})$, $\phi(\lambda+\rho)\mid_{\widetilde{A_1}}=(0,0)$. By using RS algorithm, we have $a(w_{\lambda})=1$ and $\operatorname{GKdim} L(\lambda)=5=\dim(\fu)$.
			
			By Lemma \ref{reducible} we can obtain that $z=-\frac{5}{2}$ is an irreducible point. 
			
		\end{enumerate}
		\item If $\lambda$ is non-integral and non-half-integral, one can easily check that $\Delta_{[\lambda]}\simeq  A_1$. The simple system of $\Delta_{[\lambda]}$ is $\Pi=\{\beta_1=(-2,1,1)\}$. We can get $a(w_{\lambda})=1$ by using RS algorithm, thus $\operatorname{GKdim} L(\lambda)=5=\dim(\fu)$. In this case $M_I(z\widehat{\xi_2})$ is irreducible.

		All in all, when $p=2$, $ M_I(z\widehat{\xi_2})$ is reducible if and only if
		\begin{align*}
			z&\in \{-\frac{3}{2}+\mathbb{Z}_{\geq 0}\}\cup \{\mathbb{Z}_{\geq 0}\}.
		\end{align*}
		
	\end{enumerate}
	
	So far, we have completed the proof of all cases of  type $G_2$.
\end{proof}

\begin{Thm}\label{reducible_f4}
	Let $\fg$ be of type $F_4$, then $M_I(z\widehat{\xi_p})$ is reducible if and only if
	\begin{enumerate}
		\item $p=1$, $z\in \bigcup\limits_{i=5}^{9}\{-1+\frac{1}{i}\mathbb{Z}_{\geq 0}\}$; or
		
		\item $p=2$, $z \in \bigcup\limits_{i=5}^{8}\{-\frac{i+1}{i}+\frac{1}{i}\mathbb{Z}_{\geq 0}\}$; or
		
		\item $p=3$, $z \in \bigcup\limits_{i=4}^{5}\{-\frac{i+1}{i}+\frac{1}{i}\mathbb{Z}_{\geq 0}\}\cup\bigcup\limits_{i=6}^{7}\{-\frac{i+3}{i}+\frac{1}{i}\mathbb{Z}_{\geq 0}\}$; or
		
		\item $p=4$, $z\in\bigcup\limits_{i=5}^{6}\{-1+\frac{1}{i}\mathbb{Z}_{\geq 0}\}\cup\bigcup\limits_{i=7}^{9}\{-\frac{i+1}{i}+\frac{1}{i}\mathbb{Z}_{\geq 0}\}$.
	\end{enumerate}
\end{Thm}	
\begin{proof}
	
	{\bf Step 1.} Let $\Delta^+(\fl)=\{\alpha_1\}=\{(\frac{1}{2},-\frac{1}{2},-\frac{1}{2},-\frac{1}{2})\}$. When $M_I(\lambda)$ is of scalar type, we know that $\lambda=z(\xi_2+\xi_3+\xi_4)=z(\rho-\xi_1)$ for some $z\in\mathbb{R}$, where $\xi_1$ is the fundamental weight of simple root $\alpha_1$ and $\xi_1=(1,0,0,0)$.
	
	In \cite{Knapp} we know that $\rho=(\frac{11}{2},\frac{5}{2},\frac{3}{2},\frac{1}{2})$ and $\lambda+\rho=(\frac{11}{2}+\frac{9}{2}z,\frac{5}{2}+\frac{5}{2}z,\frac{3}{2}+\frac{3}{2}z,\frac{1}{2}+\frac{1}{2}z)$.
	\begin{enumerate}
		\item If $\lambda$ is integral, then we have the following.
		\begin{enumerate}
			\item When $z=-1$, then $\lambda+\rho=(1,0,0,0)$, we can get $w_\lambda=[0, 1, 2, 1, 3, 2, 1, 0,$\\ $ 1, 2, 1, 3, 2, 1, 0]$. By using PyCox we can get that $\operatorname{GKdim} L(\lambda)=21<\dim(\fu)$.
			
			By Lemma \ref{reducible} we can obtain that $z=-1$ is a diagonal-reducible point. 
			\item When $z=-2$, then $\lambda+\rho=(-\frac{7}{2},-\frac{5}{2},-\frac{3}{2},-\frac{1}{2})$, we can get $w_\lambda=[0]$. By using PyCox we can get that $\operatorname{GKdim} L(\lambda)=23=\dim(\fu)$.
			
			By Lemma \ref{reducible} we can obtain that $z=-2$ is an irreducible point. 
			
		\end{enumerate}
		\item If $\lambda$ is half-integral, we can easily check that $\Delta_{[\lambda]}=\Delta_1\bigcup\Delta_2$, where the simple system of $\Delta_1$ is $\Pi_1=\{\beta_1=(0,0,1,1),\beta_2=(0,1,0,-1),\beta_3=(\frac{1}{2},-\frac{1}{2},-\frac{1}{2},-\frac{1}{2})\}$, and the simple system of $\Delta_2$ is $\Pi_2=\{\gamma_1=(\frac{1}{2},\frac{1}{2},-\frac{1}{2},\frac{1}{2})\}$. The Dynkin diagram of $\Delta_{[\lambda]}$ is:
		
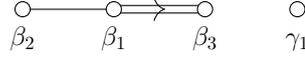
\begin{Large}
			\begin{center}
				$\dynkin[labels={\beta_2,\beta_1,\beta_3}]B3 \ \ \ \ \dynkin[labels={\gamma_1}]A1$
			\end{center}
		\end{Large}
		and with the comparison with Dynkin diagrams of classical Lie algebra, we can easily get the isomorphism $\phi:\Delta_{[\lambda]}\to B_3\times \widetilde{ A_1}$. Then $\phi(\lambda+\rho)|_{B_3}=(4z+\frac{9}{2},2z+\frac{5}{2},\frac{1}{2})$, $\phi(\lambda+\rho)|_{\widetilde{ A_1}}=(3z+\frac{7}{2},-3z-\frac{7}{2})$. By using RS algorithm, we can easily get that:
		\begin{enumerate}
			\item When $z=-\frac{1}{2}$, $a(w_{\lambda})=10$, and $\operatorname{GKdim} L(\lambda)=14<\dim(\fu)$. By Lemma \ref{reducible}, we can obtain that this point is an reducible point.
			\item When $z=-\frac{3}{2}$, $a(w_{\lambda})=1$, and $\operatorname{GKdim} L(\lambda)=23=\dim(\fu)$. By Lemma \ref{reducible}, we can obtain that this point is an irreducible point.
		\end{enumerate} 
		\item If $\lambda$ is one-third-integral, we can easily check that the simple system of $\Delta_{[\lambda]}$ is $\Pi=\{\beta_1=(0,0,1,0),\beta_2=(0,1,0,1),\beta_3=(\frac{1}{2},-\frac{1}{2},-\frac{1}{2},-\frac{1}{2})\}$. The Dynkin diagram of $\Delta_{[\lambda]}$ is:
		
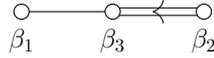
\begin{Large}
			\begin{center}
				$\dynkin[labels={\beta_1,\beta_3,\beta_2}]C3$
			\end{center}
		\end{Large}
		and with the comparison with Dynkin diagrams of classical Lie algebra, we can easily get the isomorphism $\phi:\Delta_{[\lambda]}\to C_3$. Then $\phi(\lambda+\rho)|_{C_3}=(6z+7,3z+4,3z+3)$. By using RS algorithm, we can easily get that:
		\begin{enumerate}
			\item When $z\in\{-\frac{2}{3},-\frac{1}{3}\}$, $a(w_{\lambda})=9$, and $\operatorname{GKdim} L(\lambda)=15<\dim(\fu)$. By Lemma \ref{reducible}, we can obtain that these points are reducible points.
			\item When $z\in\{-\frac{5}{3},-\frac{4}{3}\}$, $a(w_{\lambda})=1$, and $\operatorname{GKdim} L(\lambda)=23=\dim(\fu)$. By Lemma \ref{reducible}, we can obtain that these points are irreducible points.
		\end{enumerate} 
		\item If $\lambda$ is one-fourth-integral, we can easily check that $\Delta_{[\lambda]}=\Delta_1\bigcup\Delta_2$, where the simple system of $\Delta_1$ is $\Pi_1=\{\beta_1=(0,1,1,0),\beta_2=(\frac{1}{2},-\frac{1}{2},-\frac{1}{2},-\frac{1}{2})\}$, and the simple system of $\Delta_2$ is $\Pi_2=\{\gamma_1=(\frac{1}{2},-\frac{1}{2},\frac{1}{2},\frac{1}{2})\}$. The Dynkin diagram of $\Delta_{[\lambda]}$ is:
		
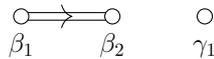
\begin{Large}
			\begin{center}
				$\dynkin[labels={\beta_1,\beta_2}]B2 \ \ \ \ \dynkin[labels={\gamma_1}]A1$
			\end{center}
		\end{Large}
		and with the comparison with Dynkin diagrams of classical Lie algebra, we can easily get the isomorphism $\phi:\Delta_{[\lambda]}\to B_2\times \widetilde{ A_1}$. Then $\phi(\lambda+\rho)|_{B_2}=(4z+\frac{9}{2},\frac{1}{2})$, $\phi(\lambda+\rho)|_{\widetilde{ A_1}}=(2z+\frac{5}{2},-2z-\frac{5}{2})$. By using RS algorithm, we can easily get that:
		\begin{enumerate}
			\item When $z\in\{-\frac{3}{4},-\frac{1}{4}\}$, $a(w_{\lambda})=5$, and $\operatorname{GKdim} L(\lambda)=19<\dim(\fu)$. By Lemma \ref{reducible}, we can obtain that these points are reducible points.
			\item When $z\in\{-\frac{7}{4},-\frac{5}{4}\}$, $a(w_{\lambda})=1$, and $\operatorname{GKdim} L(\lambda)=23=\dim(\fu)$. By Lemma \ref{reducible}, we can obtain that these points are irreducible points.
		\end{enumerate} 
		\item If $\lambda$ is one-fifth-integral, we can easily check that $\Delta_{[\lambda]}=\Delta_1\bigcup\Delta_2$, where the simple system of $\Delta_1$ is $\Pi_1=\{\beta_1=(0,1,0,0),\beta_2=(\frac{1}{2},-\frac{1}{2},-\frac{1}{2},-\frac{1}{2})\}$, and the simple system of $\Delta_2$ is $\Pi_2=\{\gamma_1=(1,0,0,1)\}$. The Dynkin diagram of $\Delta_{[\lambda]}$ is:
		\begin{Large}
			\begin{center}
				$\dynkin[labels={\beta_1,\beta_2}]A2 \ \ \ \ \dynkin[labels={\gamma_1}]A1$
			\end{center}
		\end{Large}
		and with the comparison with Dynkin diagrams of classical Lie algebra, we can easily get the isomorphism $\phi:\Delta_{[\lambda]}\to \widetilde{ A_2}\times A_1$. Then $\phi(\lambda+\rho)|_{\widetilde{ A_2}}=(\frac{10}{3}z+\frac{11}{3},-\frac{5}{3}z-\frac{4}{3},-\frac{5}{3}z-\frac{7}{3})$, $\phi(\lambda+\rho)|_{A_1}=(\frac{5}{2}z+3,-\frac{5}{2}z-3)$. By using RS algorithm, we can easily get that:
		\begin{enumerate}
			\item When $z\in\{-\frac{4}{5},-\frac{3}{5},-\frac{2}{5},-\frac{1}{5}\}$, $a(w_{\lambda})=4$, and $\operatorname{GKdim} L(\lambda)=20<\dim(\fu)$. By Lemma \ref{reducible}, we can obtain that these points are reducible points.
			\item When $z\in\{-\frac{9}{5},-\frac{8}{5},-\frac{7}{5},-\frac{6}{5}\}$, $a(w_{\lambda})=1$, and $\operatorname{GKdim} L(\lambda)=23=\dim(\fu)$. By Lemma \ref{reducible}, we can obtain that these points are irreducible points.
		\end{enumerate} 
		\item If $\lambda$ is one-sixth-integral, we can easily check that $\Delta_{[\lambda]}=\Delta_1\bigcup\Delta_2\bigcup\Delta_3$, where the simple system of $\Delta_1$ is $\Pi_1=\{\beta_1=(1,0,1,0)\}$, and the simple system of $\Delta_2$ is $\Pi_2=\{\gamma_1=(\frac{1}{2},\frac{1}{2},-\frac{1}{2},\frac{1}{2})\}$, and the simple system of $\Delta_3$ is $\Pi_3=\{\delta_1=(\frac{1}{2},-\frac{1}{2},-\frac{1}{2},-\frac{1}{2})\}$. The Dynkin diagram of $\Delta_{[\lambda]}$ is:
		\begin{Large}
			\begin{center}
				$\dynkin[labels={\beta_1}]A1 \ \ \ \ \dynkin[labels={\gamma_1}]A1 \ \ \ \ \dynkin[labels={\delta_1}]A1$
			\end{center}
		\end{Large}
		and with the comparison with Dynkin diagrams of classical Lie algebra, we can easily get the isomorphism $\phi:\Delta_{[\lambda]}\to A_1\times \widetilde{ A_1}\times \widetilde{ A_1}$. Then $\phi(\lambda+\rho)|_{A_1}=(3z+\frac{7}{2},-3z-\frac{7}{2})$, $\phi(\lambda+\rho)|_{\widetilde{ A_1}}=(3z+\frac{7}{2},-3z-\frac{7}{2})$, $\phi(\lambda+\rho)|_{\widetilde{ A_1}}=(\frac{1}{2},-\frac{1}{2})$. By using RS algorithm, we can easily get that:
		\begin{enumerate}
			\item When $z\in\{-\frac{5}{6},-\frac{1}{6}\}$, $a(w_{\lambda})=3$, and $\operatorname{GKdim} L(\lambda)=21<\dim(\fu)$. By Lemma \ref{reducible}, we can obtain that these points are reducible points.
			\item When $z\in\{-\frac{11}{6},-\frac{7}{6}\}$, $a(w_{\lambda})=1$, and $\operatorname{GKdim} L(\lambda)=23=\dim(\fu)$. By Lemma \ref{reducible}, we can obtain that these points are irreducible points.
		\end{enumerate} 
		\item If $\lambda$ is one-seventh-integral, we can easily check that $\Delta_{[\lambda]}=\Delta_1\bigcup\Delta_2$, where the simple system of $\Delta_1$ is $\Pi_1=\{\beta_1=(1,1,0,0)\}$, and the simple system of $\Delta_2$ is $\Pi_2=\{\gamma_1=(\frac{1}{2},-\frac{1}{2},-\frac{1}{2},-\frac{1}{2})\}$. The Dynkin diagram of $\Delta_{[\lambda]}$ is:
		\begin{Large}
			\begin{center}
				$\dynkin[labels={\beta_1}]A1 \ \ \ \ \dynkin[labels={\gamma_1}]A1$
			\end{center}
		\end{Large}
		and with the comparison with Dynkin diagrams of classical Lie algebra, we can easily get the isomorphism $\phi:\Delta_{[\lambda]}\to A_1\times \widetilde{ A_1}$. Then $\phi(\lambda+\rho)|_{ A_1}=(\frac{7}{2}z+4,-\frac{7}{2}z-4)$, $\phi(\lambda+\rho)|_{\widetilde{ A_1}}=(\frac{1}{2},-\frac{1}{2})$. By using RS algorithm, we can easily get that:
		\begin{enumerate}
			\item When $z\in\{-\frac{6}{7},-\frac{5}{7},-\frac{4}{7},-\frac{3}{7},-\frac{2}{7},-\frac{1}{7}\}$, $a(w_{\lambda})=2$, and $\operatorname{GKdim} L(\lambda)=22<\dim(\fu)$. By Lemma \ref{reducible}, we can obtain that these points are reducible points.
			\item When $z\in\{-\frac{13}{7},-\frac{12}{7},-\frac{11}{7},-\frac{10}{7},-\frac{9}{7},-\frac{8}{7}\}$, $a(w_{\lambda})=1$, and $\operatorname{GKdim} L(\lambda)=23=\dim(\fu)$. By Lemma \ref{reducible}, we can obtain that these points are irreducible points.
		\end{enumerate} 
		\item If $\lambda$ is one-eighth-integral, we can easily check that $\Delta_{[\lambda]}=\Delta_1\bigcup\Delta_2$, where the simple system of $\Delta_1$ is $\Pi_1=\{\beta_1=(\frac{1}{2},\frac{1}{2},\frac{1}{2},-\frac{1}{2})\}$, and the simple system of $\Delta_2$ is $\Pi_2=\{\gamma_1=(\frac{1}{2},-\frac{1}{2},-\frac{1}{2},-\frac{1}{2})\}$. The Dynkin diagram of $\Delta_{[\lambda]}$ is:
		\begin{Large}
			\begin{center}
				$\dynkin[labels={\beta_1}]A1 \ \ \ \ \dynkin[labels={\gamma_1}]A1$
			\end{center}
		\end{Large}
		and with the comparison with Dynkin diagrams of classical Lie algebra, we can easily get the isomorphism $\phi:\Delta_{[\lambda]}\to \widetilde{ A_1}\times \widetilde{ A_1}$. Then $\phi(\lambda+\rho)|_{\widetilde{ A_1}}=(4z+\frac{9}{2},-4z-\frac{9}{2})$, $\phi(\lambda+\rho)|_{\widetilde{ A_1}}=(\frac{1}{2},-\frac{1}{2})$. By using RS algorithm, we can easily get that:
		\begin{enumerate}
			\item When $z\in\{-\frac{7}{8},-\frac{5}{8},-\frac{3}{8},-\frac{1}{8}\}$, $a(w_{\lambda})=2$, and $\operatorname{GKdim} L(\lambda)=22<\dim(\fu)$. By Lemma \ref{reducible}, we can obtain that these points are reducible points.
			\item When $z\in\{-\frac{15}{8},-\frac{13}{8},-\frac{11}{8},-\frac{9}{8}\}$, $a(w_{\lambda})=1$, and $\operatorname{GKdim} L(\lambda)=23=\dim(\fu)$. By Lemma \ref{reducible}, we can obtain that these points are irreducible points.
		\end{enumerate} 
		\item If $\lambda$ is one-ninth-integral, we can easily check that the simple system of $\Delta_{[\lambda]}$ is $\Pi=\{\beta_1=(\frac{1}{2},\frac{1}{2},\frac{1}{2},\frac{1}{2}),\beta_2=(\frac{1}{2},-\frac{1}{2},-\frac{1}{2},-\frac{1}{2})\}$. The Dynkin diagram of $\Delta_{[\lambda]}$ is:
		\begin{Large}
			\begin{center}
				$\dynkin[labels={\beta_1,\beta_2}]A2$
			\end{center}
		\end{Large}
		and with the comparison with Dynkin diagrams of classical Lie algebra, we can easily get the isomorphism $\phi:\Delta_{[\lambda]}\to \widetilde{ A_2}$. Then $\phi(\lambda+\rho)|_{\widetilde{ A_2}}=(6z+7,-3z-3,-3z-4)$. By using RS algorithm, we can easily get that:
		\begin{enumerate}
			\item When $z\in\{-\frac{8}{9},-\frac{7}{9},-\frac{5}{9},-\frac{4}{9},-\frac{2}{9},-\frac{1}{9}\}$, $a(w_{\lambda})=3$, and $\operatorname{GKdim} L(\lambda)=21<\dim(\fu)$. By Lemma \ref{reducible}, we can obtain that these points are reducible points.
			\item When $z\in\{-\frac{17}{9},-\frac{16}{9},-\frac{14}{9},-\frac{13}{9},-\frac{11}{9},-\frac{10}{9}\}$, $a(w_{\lambda})=1$, and $\operatorname{GKdim} L(\lambda)=23=\dim(\fu)$. By Lemma \ref{reducible}, we can obtain that these points are irreducible points.
		\end{enumerate} 
		\item If $\lambda$ is not belonging to any of the types mentioned above, one can easily check that $\Delta_{[\lambda]}\simeq  A_1$. The simple system of $\Delta_{[\lambda]}$ is $\Pi=\{\beta_1=(\frac{1}{2},-\frac{1}{2},-\frac{1}{2},-\frac{1}{2})\}$. We can get $a(w_{\lambda})=1$ by using RS algorithm, thus $\operatorname{GKdim} L(\lambda)=23=\dim(\fu)$. In this case $M_I(z\widehat{\xi_1})$ is irreducible.
		
		All in all, when $p=1$, $ M_I(z\widehat{\xi_1})$ is reducible if and only if
		\begin{align*}
			\begin{split}
				z&\in \{-1+\mathbb{Z}_{\geq 0}\}
				\\&\cup \{-\frac{1}{2}+\mathbb{Z}_{\geq 0}\} \\&\cup \{-\frac{2}{3}+\mathbb{Z}_{\geq 0}\} \cup \{-\frac{1}{3}+\mathbb{Z}_{\geq 0}\} \\&\cup \{-\frac{3}{4}+\mathbb{Z}_{\geq 0}\} \cup \{-\frac{1}{4}+\mathbb{Z}_{\geq 0}\} \\&\cup \{-\frac{4}{5}+\mathbb{Z}_{\geq 0}\} \cup \{-\frac{3}{5}+\mathbb{Z}_{\geq 0}\} \cup \{-\frac{2}{5}+\mathbb{Z}_{\geq 0}\} \cup \{-\frac{1}{5}+\mathbb{Z}_{\geq 0}\} \\&\cup \{-\frac{5}{6}+\mathbb{Z}_{\geq 0}\} \cup \{-\frac{1}{6}+\mathbb{Z}_{\geq 0}\} \\&\cup \{-\frac{6}{7}+\mathbb{Z}_{\geq 0}\} \cup \{-\frac{5}{7}+\mathbb{Z}_{\geq 0}\} \cup \{-\frac{4}{7}+\mathbb{Z}_{\geq 0}\} \cup \{-\frac{3}{7}+\mathbb{Z}_{\geq 0}\} \cup \{-\frac{2}{7}+\mathbb{Z}_{\geq 0}\} \cup \{-\frac{1}{7}+\mathbb{Z}_{\geq 0}\} \\&\cup \{-\frac{7}{8}+\mathbb{Z}_{\geq 0}\} \cup \{-\frac{5}{8}+\mathbb{Z}_{\geq 0}\} \cup \{-\frac{3}{8}+\mathbb{Z}_{\geq 0}\} \cup \{-\frac{1}{8}+\mathbb{Z}_{\geq 0}\} \\&\cup \{-\frac{8}{9}+\mathbb{Z}_{\geq 0}\} \cup \{-\frac{7}{9}+\mathbb{Z}_{\geq 0}\} \cup \{-\frac{5}{9}+\mathbb{Z}_{\geq 0}\} \cup \{-\frac{4}{9}+\mathbb{Z}_{\geq 0}\} \cup \{-\frac{2}{9}+\mathbb{Z}_{\geq 0}\} \cup \{-\frac{1}{9}+\mathbb{Z}_{\geq 0}\}\\&=\bigcup\limits_{i=5}^{9}\{-1+\frac{1}{i}\mathbb{Z}_{\geq 0}\}. 
			\end{split}
		\end{align*}
		
	\end{enumerate}

				
	
	{\bf Step 2.} Let $\Delta^+(\fl)=\{(0,0,0,1)\}=\{\alpha_2\}$. When $M_I(\lambda)$ is of scalar type, we know that $\lambda=z(\xi_1+\xi_3+\xi_4)=z(\rho-\xi_2)$ for some $z\in\mathbb{R}$, where $\xi_2$ is the fundamental weight of simple root $\alpha_2$ and $\xi_2=(\frac{3}{2},\frac{1}{2},\frac{1}{2},\frac{1}{2})$.
	In \cite{Knapp} we know that $\rho=(\frac{11}{2},\frac{5}{2},\frac{3}{2},\frac{1}{2})$ and $\lambda+\rho=(\frac{11}{2}+4z,\frac{5}{2}+2z,\frac{3}{2}+z,\frac{1}{2})$.
	

	\begin{enumerate}
		\item If $\lambda$ is integral, then we have the following.
		\begin{enumerate}
			\item When $z=-1$, then $\lambda+\rho=(\frac{3}{2},\frac{1}{2},\frac{1}{2},\frac{1}{2})$, we can get $w_\lambda=[1, 2, 1, 0, 1, 3, 2, 1, \\2, 0, 1, 3, 2, 1, 2, 0, 1, 3, 2, 1]$. By using PyCox we can get that $\operatorname{GKdim} L(\lambda)=20<\dim(\fu)$.
			By Lemma \ref{reducible} we can obtain that $z=-1$ is a diagonal-reducible point.
			\item When $z=-2$, then $\lambda+\rho=(-\frac{5}{2},-\frac{3}{2},-\frac{1}{2},\frac{1}{2})$, we can get $w_\lambda=[1]$. By using PyCox we can get that $\operatorname{GKdim} L(\lambda)=23=\dim(\fu)$.
			By Lemma \ref{reducible} we can obtain that $z=-2$ is an irreducible point.
		\end{enumerate}
		\item If $\lambda$ is half-integral, we can easily check that $\Delta_{[\lambda]}=\Delta_1\bigcup\Delta_2$, where the simple system of $\Delta_1$ is $\Pi_1=\{\beta_1=(0,0,1,0)\}$, and the simple system of $\Delta_2$ is $\Pi_2=\{\gamma_1=(0,0,0,1),\gamma_2=(1,-1,0,0),\gamma_3=(0,1,0,-1)\}$. The Dynkin diagram of $\Delta_{[\lambda]}$ is:
		
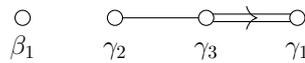
\begin{Large}
			\begin{center}
				$\dynkin[labels={\beta_1}]A1 \ \ \ \ \dynkin[labels={\gamma_2,\gamma_3,\gamma_1}]B3$
			\end{center}
		\end{Large}
		and with the comparison with Dynkin diagrams of classical Lie algebra, we can easily get the isomorphism $\phi:\Delta_{[\lambda]}\to \widetilde{ A_1}\times B_3$. Then $\phi(\lambda+\rho)|_{\widetilde{ A_1}}=(z+\frac{3}{2},-z-\frac{3}{2})$, $\phi(\lambda+\rho)|_{B_3}=(4z+\frac{11}{2},2z+\frac{5}{2},\frac{1}{2})$. By using RS algorithm, we can easily get that:
		\begin{enumerate}
			\item When $z=-\frac{1}{2}$, $a(w_{\lambda})=10$, and $\operatorname{GKdim} L(\lambda)=14<\dim(\fu)$. By Lemma \ref{reducible}, we can obtain that this point is an reducible point.
			\item When $z=-\frac{3}{2}$, $a(w_{\lambda})=1$, and $\operatorname{GKdim} L(\lambda)=23=\dim(\fu)$. By Lemma \ref{reducible}, we can obtain that this point is an irreducible point.
		\end{enumerate} 
		\item If $\lambda$ is one-third-integral, we can easily check that $\Delta_{[\lambda]}=\Delta_1\bigcup\Delta_2$, where the simple system of $\Delta_1$ is $\Pi_1=\{\beta_1=(0,0,0,1),\beta_2=(\frac{1}{2},-\frac{1}{2},\frac{1}{2},-\frac{1}{2})\}$, and the simple system of $\Delta_2$ is $\Pi_2=\{\gamma_1=(0,1,1,0),\gamma_2=(1,0,-1,0)\}$. The Dynkin diagram of $\Delta_{[\lambda]}$ is:
		\begin{Large}
			\begin{center}
				$\dynkin[labels={\beta_1,\beta_2}]A2 \ \ \ \ \dynkin[labels={\gamma_1,\gamma_2}]A2$
			\end{center}
		\end{Large}
		and with the comparison with Dynkin diagrams of classical Lie algebra, we can easily get the isomorphism $\phi:\Delta_{[\lambda]}\to \widetilde{ A_2}\times A_2$. Then $\phi(\lambda+\rho)|_{\widetilde{ A_2}}=(z+2,z+1,-2z-3)$, $\phi(\lambda+\rho)|_{ A_2}=(3z+4,0,-3z-4)$. By using RS algorithm, we can easily get that:
		\begin{enumerate}
			\item When $z\in\{-\frac{2}{3},-\frac{1}{3}\}$, $a(w_{\lambda})=6$, and $\operatorname{GKdim} L(\lambda)=18<\dim(\fu)$. By Lemma \ref{reducible}, we can obtain that these points are reducible points.
			\item When $z\in\{-\frac{5}{3},-\frac{4}{3}\}$, $a(w_{\lambda})=1$, and $\operatorname{GKdim} L(\lambda)=23=\dim(\fu)$. By Lemma \ref{reducible}, we can obtain that these points are irreducible points.
		\end{enumerate} 
		\item If $\lambda$ is one-fourth-integral, we can easily check that $\Delta_{[\lambda]}=\Delta_1\bigcup\Delta_2$, where the simple system of $\Delta_1$ is $\Pi_1=\{\beta_1=(0,1,0,0)\}$, and the simple system of $\Delta_2$ is $\Pi_2=\{\gamma_1=(0,0,0,1),\gamma_2=(1,0,0,-1)\}$. The Dynkin diagram of $\Delta_{[\lambda]}$ is:
		\begin{Large}
			\begin{center}
				$\dynkin[labels={\beta_1}]A1 \ \ \ \ \dynkin[labels={\gamma_2,\gamma_1}]B2$
			\end{center}
		\end{Large}
		and with the comparison with Dynkin diagrams of classical Lie algebra, we can easily get the isomorphism $\phi:\Delta_{[\lambda]}\to \widetilde{ A_1}\times B_2$. Then $\phi(\lambda+\rho)|_{\widetilde{ A_1}}=(2z+\frac{5}{2},-2z-\frac{5}{2})$, $\phi(\lambda+\rho)|_{B_2}=(4z+\frac{11}{2},\frac{1}{2})$. By using RS algorithm, we can easily get that:
		\begin{enumerate}
			\item When $z\in\{-\frac{3}{4},-\frac{1}{4}\}$, $a(w_{\lambda})=5$, and $\operatorname{GKdim} L(\lambda)=19<\dim(\fu)$. By Lemma \ref{reducible}, we can obtain that these points are reducible points.
			\item When $z\in\{-\frac{7}{4},-\frac{5}{4}\}$, $a(w_{\lambda})=1$, and $\operatorname{GKdim} L(\lambda)=23=\dim(\fu)$. By Lemma \ref{reducible}, we can obtain that these points are irreducible points.
		\end{enumerate} 
		\item If $\lambda$ is one-fifth-integral, we can easily check that $\Delta_{[\lambda]}=\Delta_1\bigcup\Delta_2$, where the simple system of $\Delta_1$ is $\Pi_1=\{\beta_1=(0,0,0,1),\beta_2=(\frac{1}{2},\frac{1}{2},-\frac{1}{2},-\frac{1}{2})\}$, and the simple system of $\Delta_2$ is $\Pi_2=\{\gamma_1=(1,0,1,0)\}$. The Dynkin diagram of $\Delta_{[\lambda]}$ is:
		\begin{Large}
			\begin{center}
				$\dynkin[labels={\beta_1,\beta_2}]A2 \ \ \ \ \dynkin[labels={\gamma_1}]A1$
			\end{center}
		\end{Large}
		and with the comparison with Dynkin diagrams of classical Lie algebra, we can easily get the isomorphism $\phi:\Delta_{[\lambda]}\to \widetilde{ A_2}\times A_1$. Then $\phi(\lambda+\rho)|_{\widetilde{ A_2}}=(\frac{5}{3}z+\frac{8}{3},\frac{5}{3}z+\frac{5}{3},-\frac{10}{3}z-\frac{13}{3})$, $\phi(\lambda+\rho)|_{A_1}=(\frac{5}{2}z+\frac{7}{2},-\frac{5}{2}z-\frac{7}{2})$. By using RS algorithm, we can easily get that:
		\begin{enumerate}
			\item When $z\in\{-\frac{4}{5},-\frac{3}{5},-\frac{2}{5}\}$, $a(w_{\lambda})=4$, and $\operatorname{GKdim} L(\lambda)=20<\dim(\fu)$. By Lemma \ref{reducible}, we can obtain that these points are reducible points.
			\item When $z=-\frac{6}{5}$, $a(w_{\lambda})=2$, and $\operatorname{GKdim} L(\lambda)=22<\dim(\fu)$. By Lemma \ref{reducible}, we can obtain that this point is an reducible point.
			\item When $z\in\{-\frac{9}{5},-\frac{8}{5},-\frac{7}{5},-\frac{11}{5}\}$, $a(w_{\lambda})=1$, and $\operatorname{GKdim} L(\lambda)=23=\dim(\fu)$. By Lemma \ref{reducible}, we can obtain that these points are irreducible points.
		\end{enumerate} 
		\item If $\lambda$ is one-sixth-integral, we can easily check that $\Delta_{[\lambda]}=\Delta_1\bigcup\Delta_2$, where the simple system of $\Delta_1$ is $\Pi_1=\{\beta_1=(0,0,0,1)\}$, and the simple system of $\Delta_2$ is $\Pi_2=\{\gamma_1=(1,1,0,0)\}$. The Dynkin diagram of $\Delta_{[\lambda]}$ is:
		\begin{Large}
			\begin{center}
				$\dynkin[labels={\beta_1}]A1 \ \ \ \ \dynkin[labels={\gamma_1}]A1$
			\end{center}
		\end{Large}
		and with the comparison with Dynkin diagrams of classical Lie algebra, we can easily get the isomorphism $\phi:\Delta_{[\lambda]}\to \widetilde{ A_1}\times A_1$. Then $\phi(\lambda+\rho)|_{\widetilde{ A_1}}=(\frac{1}{2},-\frac{1}{2})$, $\phi(\lambda+\rho)|_{ A_1}=(3z+4,-3z-4)$. By using RS algorithm, we can easily get that:
		\begin{enumerate}
			\item When $z\in\{-\frac{5}{6},-\frac{7}{6}\}$, $a(w_{\lambda})=2$, and $\operatorname{GKdim} L(\lambda)=22<\dim(\fu)$. By Lemma \ref{reducible}, we can obtain that these points are reducible points.
			\item When $z\in\{-\frac{11}{6},-\frac{13}{6}\}$, $a(w_{\lambda})=1$, and $\operatorname{GKdim} L(\lambda)=23=\dim(\fu)$. By Lemma \ref{reducible}, we can obtain that these points are irreducible points.
		\end{enumerate} 
		\item If $\lambda$ is one-seventh-integral, we can easily check that the simple system of $\Delta_{[\lambda]}$ is $\Pi=\{\beta_1=(0,0,0,1),\beta_2=(\frac{1}{2},\frac{1}{2},\frac{1}{2},-\frac{1}{2})\}$. The Dynkin diagram of $\Delta_{[\lambda]}$ is:
		\begin{Large}
			\begin{center}
				$\dynkin[labels={\beta_1,\beta_2}]A2$
			\end{center}
		\end{Large}
		and with the comparison with Dynkin diagrams of classical Lie algebra, we can easily get the isomorphism $\phi:\Delta_{[\lambda]}\to \widetilde{ A_2}$. Then $\phi(\lambda+\rho)|_{\widetilde{ A_2}}=(\frac{7}{3}z+\frac{11}{3},\frac{7}{3}z+\frac{8}{3},-\frac{14}{3}z-\frac{19}{3})$. By using RS algorithm, we can easily get that:
		\begin{enumerate}
			\item When $z\in\{-\frac{6}{7},-\frac{5}{7},-\frac{4}{7},-\frac{3}{7},-\frac{2}{7},-\frac{8}{7}\}$, $a(w_{\lambda})=3$, and $\operatorname{GKdim} L(\lambda)=21<\dim(\fu)$. By Lemma \ref{reducible}, we can obtain that these points are reducible points.
			\item When $z\in\{-\frac{13}{7},-\frac{12}{7},-\frac{11}{7},-\frac{10}{7},-\frac{9}{7},-\frac{15}{7}\}$, $a(w_{\lambda})=1$, and $\operatorname{GKdim} L(\lambda)=23=\dim(\fu)$. By Lemma \ref{reducible}, we can obtain that these points are irreducible points.
		\end{enumerate} 
		\item If $\lambda$ is one-eighth-integral, we can easily check that $\Delta_{[\lambda]}=\Delta_1\bigcup\Delta_2$, where the simple system of $\Delta_1$ is $\Pi_1=\{\beta_1=(1,0,0,0)\}$, and the simple system of $\Delta_2$ is $\Pi_2=\{\gamma_1=(0,0,0,1)\}$. The Dynkin diagram of $\Delta_{[\lambda]}$ is:
		\begin{Large}
			\begin{center}
				$\dynkin[labels={\beta_1}]A1 \ \ \ \ \dynkin[labels={\gamma_1}]A1$
			\end{center}
		\end{Large}
		and with the comparison with Dynkin diagrams of classical Lie algebra, we can easily get the isomorphism $\phi:\Delta_{[\lambda]}\to \widetilde{ A_1}\times \widetilde{ A_1}$. Then $\phi(\lambda+\rho)|_{\widetilde{ A_1}}=(4z+\frac{11}{2},-4z-\frac{11}{2})$, $\phi(\lambda+\rho)|_{\widetilde{ A_1}}=(\frac{1}{2},-\frac{1}{2})$. By using RS algorithm, we can easily get that:
		\begin{enumerate}
			\item When $z\in\{-\frac{7}{8},-\frac{5}{8},-\frac{3}{8},-\frac{9}{8}\}$, $a(w_{\lambda})=2$, and $\operatorname{GKdim} L(\lambda)=22<\dim(\fu)$. By Lemma \ref{reducible}, we can obtain that these points are reducible points.
			\item When $z\in\{-\frac{15}{8},-\frac{13}{8},-\frac{11}{8},-\frac{17}{8}\}$, $a(w_{\lambda})=1$, and $\operatorname{GKdim} L(\lambda)=23=\dim(\fu)$. By Lemma \ref{reducible}, we can obtain that these points are irreducible points.
		\end{enumerate} 
		\item If $\lambda$ is not belonging to any of the types mentioned above, one can easily check that $\Delta_{[\lambda]}\simeq  A_1$. The simple system of $\Delta_{[\lambda]}$ is $\Pi=\{\beta_1=(0,0,0,1)\}$. We can get $a(w_{\lambda})=1$ by using RS algorithm, thus $\operatorname{GKdim} L(\lambda)=23=\dim(\fu)$. In this case $M_I(z\widehat{\xi_2})$ is irreducible.
		
		All in all, when $p=2$, $ M_I(z\widehat{\xi_2})$ is reducible if and only if
		\begin{align*}
			\begin{split}
				z&\in \{-1+\mathbb{Z}_{\geq 0}\}
				\\&\cup \{-\frac{1}{2}+\mathbb{Z}_{\geq 0}\} \\&\cup \{-\frac{2}{3}+\mathbb{Z}_{\geq 0}\} \cup \{-\frac{1}{3}+\mathbb{Z}_{\geq 0}\} \\&\cup \{-\frac{3}{4}+\mathbb{Z}_{\geq 0}\} \cup \{-\frac{1}{4}+\mathbb{Z}_{\geq 0}\} \\&\cup \{-\frac{4}{5}+\mathbb{Z}_{\geq 0}\} \cup \{-\frac{3}{5}+\mathbb{Z}_{\geq 0}\} \cup \{-\frac{2}{5}+\mathbb{Z}_{\geq 0}\} \cup \{-\frac{6}{5}+\mathbb{Z}_{\geq 0}\} \\&\cup \{-\frac{5}{6}+\mathbb{Z}_{\geq 0}\} \cup \{-\frac{7}{6}+\mathbb{Z}_{\geq 0}\} \\&\cup \{-\frac{6}{7}+\mathbb{Z}_{\geq 0}\} \cup \{-\frac{5}{7}+\mathbb{Z}_{\geq 0}\} \cup \{-\frac{4}{7}+\mathbb{Z}_{\geq 0}\} \cup \{-\frac{3}{7}+\mathbb{Z}_{\geq 0}\} \cup \{-\frac{2}{7}+\mathbb{Z}_{\geq 0}\} \cup \{-\frac{8}{7}+\mathbb{Z}_{\geq 0}\} \\&\cup \{-\frac{7}{8}+\mathbb{Z}_{\geq 0}\} \cup \{-\frac{5}{8}+\mathbb{Z}_{\geq 0}\} \cup \{-\frac{3}{8}+\mathbb{Z}_{\geq 0}\} \cup \{-\frac{9}{8}+\mathbb{Z}_{\geq 0}\}\\&=\bigcup\limits_{i=5}^{8}\{-\frac{i+1}{i}+\frac{1}{i}\mathbb{Z}_{\geq 0}\}. 
			\end{split}
		\end{align*}
	\end{enumerate}

	{\bf Step 3.} Let $\Delta^+(\fl)=\{(0,0,1,-1)=\{\alpha_3\}\}$. When $M_I(\lambda)$ is of scalar type, we know that $\lambda=z(\xi_1+\xi_2+\xi_4)=z(\rho-\xi_3)$ for some $z\in\mathbb{R}$, where $\xi_3$ is the fundamental weight of simple root $\alpha_3$ and $\xi_3=(2,1,1,0)$.
	In \cite{Knapp} we know that $\rho=(\frac{11}{2},\frac{5}{2},\frac{3}{2},\frac{1}{2})$ and $\lambda+\rho=(\frac{11}{2}+\frac{7}{2}z,\frac{5}{2}+\frac{3}{2}z,\frac{3}{2}+\frac{1}{2}z,\frac{1}{2}+\frac{1}{2}z)$.
	\begin{enumerate}
		\item If $\lambda$ is integral, then we have the following.
		\begin{enumerate}
			\item When $z=-1$, then $\lambda+\rho=(2,1,1,0)$, we can get $w_\lambda=[2, 3, 1, 2, 3, 1, 2, 0, \\1, 2, 1, 0, 3, 2, 1, 2, 3, 0, 1, 2]$. By using PyCox we can get that $\operatorname{GKdim} L(\lambda)=20<\dim(\fu)$.
			By Lemma \ref{reducible} we can obtain that $z=-1$ is a diagonal-reducible point.
			\item When $z=-2$, then $\lambda+\rho=(-\frac{3}{2},-\frac{1}{2},\frac{1}{2},-\frac{1}{2})$, we can get $w_\lambda=[1, 2]$. By using PyCox we can get that $\operatorname{GKdim} L(\lambda)=23=\dim(\fu)$.
			By Lemma \ref{reducible} we can obtain that $z=-2$ is an irreducible point.
		\end{enumerate}
		\item If $\lambda$ is half-integral, we can easily check that $\Delta_{[\lambda]}=\Delta_1\bigcup\Delta_2$, where the simple system of $\Delta_1$ is $\Pi_1=\{\beta_1=(0,1,0,1),\beta_2=(0,0,1,-1),\beta_3=(\frac{1}{2},-\frac{1}{2},-\frac{1}{2},\frac{1}{2})\}$, and the simple system of $\Delta_2$ is $\Pi_2=\{\gamma_1=(\frac{1}{2},\frac{1}{2},-\frac{1}{2},-\frac{1}{2})\}$. The Dynkin diagram of $\Delta_{[\lambda]}$ is:
		
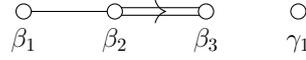
\begin{Large}
			\begin{center}
				$\dynkin[labels={\beta_1,\beta_2,\beta_3}]B3 \ \ \ \ \dynkin[labels={\gamma_1}]A1$
			\end{center}
		\end{Large}
		and with the comparison with Dynkin diagrams of classical Lie algebra, we can easily get the isomorphism $\phi:\Delta_{[\lambda]}\to B_3\times \widetilde{ A_1}$. Then $\phi(\lambda+\rho)|_{B_3}=(3z+5,z+2,z+1)$, $\phi(\lambda+\rho)|_{\widetilde{ A_1}}=(2z+3,-2z-3)$. By using RS algorithm, we can easily get that:
		\begin{enumerate}
			\item When $z=-\frac{3}{2}$, $a(w_{\lambda})=2$, and $\operatorname{GKdim} L(\lambda)=22<\dim(\fu)$. By Lemma \ref{reducible}, we can obtain that this point is an reducible point.
			\item When $z=-\frac{5}{2}$, $a(w_{\lambda})=1$, and $\operatorname{GKdim} L(\lambda)=23=\dim(\fu)$. By Lemma \ref{reducible}, we can obtain that this point is an irreducible point.
		\end{enumerate} 
		\item If $\lambda$ is one-third-integral, we can easily check that $\Delta_{[\lambda]}=\Delta_1\bigcup\Delta_2$, where the simple system of $\Delta_1$ is $\Pi_1=\{\beta_1=(0,1,0,0),\beta_2=(\frac{1}{2},-\frac{1}{2},\frac{1}{2},\frac{1}{2})\}$, and the simple system of $\Delta_2$ is $\Pi_2=\{\gamma_1=(1,0,-1,0),\gamma_2=(0,0,1,-1)\}$. The Dynkin diagram of $\Delta_{[\lambda]}$ is:
		
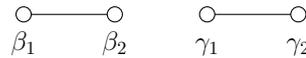
\begin{Large}
			\begin{center}
				$\dynkin[labels={\beta_1,\beta_2}]A2 \ \ \ \ \dynkin[labels={\gamma_1,\gamma_2}]A2$
			\end{center}
		\end{Large}
		and with the comparison with Dynkin diagrams of classical Lie algebra, we can easily get the isomorphism $\phi:\Delta_{[\lambda]}\to \widetilde{ A_2}\times A_2$. Then $\phi(\lambda+\rho)|_{\widetilde{ A_2}}=(3z+5,0,-3z-5)$, $\phi(\lambda+\rho)|_{ A_2}=(2z+3,-z-1,-z-2)$. By using RS algorithm, we can easily get that:
		\begin{enumerate}
			\item When $z=-\frac{2}{3}$, $a(w_{\lambda})=6$, and $\operatorname{GKdim} L(\lambda)=18<\dim(\fu)$. By Lemma \ref{reducible}, we can obtain that this point is an reducible point.
			\item When $z=-\frac{4}{3}$, $a(w_{\lambda})=4$, and $\operatorname{GKdim} L(\lambda)=20<\dim(\fu)$. By Lemma \ref{reducible}, we can obtain that this point is an reducible point.
			\item When $z\in\{-\frac{5}{3},-\frac{7}{3}\}$, $a(w_{\lambda})=1$, and $\operatorname{GKdim} L(\lambda)=23=\dim(\fu)$. By Lemma \ref{reducible}, we can obtain that these points are irreducible points.
		\end{enumerate} 
		\item If $\lambda$ is one-fourth-integral, we can easily check that $\Delta_{[\lambda]}=\Delta_1\bigcup\Delta_2$, where the simple system of $\Delta_1$ is $\Pi_1=\{\beta_1=(1,0,0,1),\beta_2=(0,0,1,-1)\}$, and the simple system of $\Delta_2$ is $\Pi_2=\{\gamma_1=(\frac{1}{2},\frac{1}{2},-\frac{1}{2},-\frac{1}{2})\}$. The Dynkin diagram of $\Delta_{[\lambda]}$ is:
		
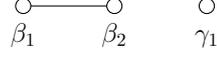
\begin{Large}
			\begin{center}
				$\dynkin[labels={\beta_1,\beta_2}]A2 \ \ \ \ \dynkin[labels={\gamma_1}]A1$
			\end{center}
		\end{Large}
		and with the comparison with Dynkin diagrams of classical Lie algebra, we can easily get the isomorphism $\phi:\Delta_{[\lambda]}\to A_2\times \widetilde{ A_1}$. Then $\phi(\lambda+\rho)|_{ A_2}=(\frac{8}{3}z+\frac{13}{3},-\frac{4}{3}z-\frac{5}{3},-\frac{4}{3}z-\frac{8}{3})$, $\phi(\lambda+\rho)|_{\widetilde{ A_1}}=(2z+3,-2z-3)$. By using RS algorithm, we can easily get that:
		\begin{enumerate}
			\item When $z\in\{-\frac{3}{4},-\frac{5}{4}\}$, $a(w_{\lambda})=4$, and $\operatorname{GKdim} L(\lambda)=20<\dim(\fu)$. By Lemma \ref{reducible}, we can obtain that these points are reducible points.
			\item When $z\in\{-\frac{7}{4},-\frac{9}{4}\}$, $a(w_{\lambda})=1$, and $\operatorname{GKdim} L(\lambda)=23=\dim(\fu)$. By Lemma \ref{reducible}, we can obtain that these points are irreducible points.
		\end{enumerate} 
		\item If $\lambda$ is one-fifth-integral, we can easily check that the simple system of $\Delta_{[\lambda]}$ is $\Pi=\{\beta_1=(0,0,1,-1),\beta_2=(\frac{1}{2},\frac{1}{2},-\frac{1}{2},\frac{1}{2})\}$. The Dynkin diagram of $\Delta_{[\lambda]}$ is:
		
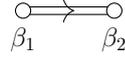
\begin{Large}
			\begin{center}
				$\dynkin[labels={\beta_1,\beta_2}]B2$
			\end{center}
		\end{Large}
		and with the comparison with Dynkin diagrams of classical Lie algebra, we can easily get the isomorphism $\phi:\Delta_{[\lambda]}\to B_2$. Then $\phi(\lambda+\rho)|_{B_2}=(\frac{5}{2}z+\frac{9}{2},\frac{5}{2}z+\frac{7}{2})$. By using RS algorithm, we can easily get that:
		\begin{enumerate}
			\item When $z\in\{-\frac{4}{5},-\frac{3}{5},-\frac{2}{5},-\frac{6}{5}\}$, $a(w_{\lambda})=4$, and $\operatorname{GKdim} L(\lambda)=20<\dim(\fu)$. By Lemma \ref{reducible}, we can obtain that these points are reducible points.
			\item When $z\in\{-\frac{9}{5},-\frac{8}{5},-\frac{7}{5},-\frac{11}{5}\}$, $a(w_{\lambda})=1$, and $\operatorname{GKdim} L(\lambda)=23=\dim(\fu)$. By Lemma \ref{reducible}, we can obtain that these points are irreducible points.
		\end{enumerate} 
		\item If $\lambda$ is one-sixth-integral, we can easily check that $\Delta_{[\lambda]}=\Delta_1\bigcup\Delta_2$, where the simple system of $\Delta_1$ is $\Pi_1=\{\beta_1=(0,0,1,-1)\}$, and the simple system of $\Delta_2$ is $\Pi_2=\{\gamma_1=(\frac{1}{2},\frac{1}{2},\frac{1}{2},\frac{1}{2})\}$. The Dynkin diagram of $\Delta_{[\lambda]}$ is:
		
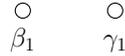
\begin{Large}
			\begin{center}
				$\dynkin[labels={\beta_1}]A1 \ \ \ \ \dynkin[labels={\gamma_1}]A1$
			\end{center}
		\end{Large}
		and with the comparison with Dynkin diagrams of classical Lie algebra, we can easily get the isomorphism $\phi:\Delta_{[\lambda]}\to A_1\times \widetilde{ A_1}$. Then $\phi(\lambda+\rho)|_{ A_1}=(\frac{1}{2},-\frac{1}{2})$, $\phi(\lambda+\rho)|_{\widetilde{ A_1}}=(3z+5,-3z-5)$. By using RS algorithm, we can easily get that:
		\begin{enumerate}
			\item When $z\in\{-\frac{5}{6},-\frac{7}{6}\}$, $a(w_{\lambda})=2$, and $\operatorname{GKdim} L(\lambda)=22<\dim(\fu)$. By Lemma \ref{reducible}, we can obtain that these points are reducible points.
			\item When $z\in\{-\frac{11}{6},-\frac{13}{6}\}$, $a(w_{\lambda})=1$, and $\operatorname{GKdim} L(\lambda)=23=\dim(\fu)$. By Lemma \ref{reducible}, we can obtain that these points are irreducible points.
		\end{enumerate} 
		\item If $\lambda$ is one-seventh-integral, we can easily check that $\Delta_{[\lambda]}=\Delta_1\bigcup\Delta_2$, where the simple system of $\Delta_1$ is $\Pi_1=\{\beta_1=(1,0,0,0)\}$, and the simple system of $\Delta_2$ is $\Pi_2=\{\gamma_1=(0,0,1,-1)\}$. The Dynkin diagram of $\Delta_{[\lambda]}$ is:
		\begin{Large}
			\begin{center}
				$\dynkin[labels={\beta_1}]A1 \ \ \ \ \dynkin[labels={\gamma_1}]A1$
			\end{center}
		\end{Large}
		and with the comparison with Dynkin diagrams of classical Lie algebra, we can easily get the isomorphism $\phi:\Delta_{[\lambda]}\to \widetilde{ A_1}\times A_1$. Then $\phi(\lambda+\rho)|_{\widetilde{ A_1}}=(\frac{7}{2}z+\frac{11}{2},-\frac{7}{2}z-\frac{11}{2})$, $\phi(\lambda+\rho)|_{ A_1}=(\frac{1}{2},-\frac{1}{2})$. By using RS algorithm, we can easily get that:
		\begin{enumerate}
			\item When $z\in\{-\frac{6}{7},-\frac{5}{7},-\frac{4}{7},-\frac{10}{7},-\frac{9}{7},-\frac{8}{7}\}$, $a(w_{\lambda})=2$, and $\operatorname{GKdim} L(\lambda)=22<\dim(\fu)$. By Lemma \ref{reducible}, we can obtain that these points are reducible points.
			\item When $z\in\{-\frac{13}{7},-\frac{12}{7},-\frac{11}{7},-\frac{17}{7},-\frac{16}{7},-\frac{15}{7}\}$, $a(w_{\lambda})=1$, and $\operatorname{GKdim} L(\lambda)=23=\dim(\fu)$. By Lemma \ref{reducible}, we can obtain that these points are irreducible points.
		\end{enumerate} 
		\item If $\lambda$ is not belonging to any of the types mentioned above, one can easily check that $\Delta_{[\lambda]}\simeq  A_1$. The simple system of $\Delta_{[\lambda]}$ is $\Pi=\{\beta_1=(0,0,1,-1)\}$. We can get $a(w_{\lambda})=1$ by using RS algorithm, thus $\operatorname{GKdim} L(\lambda)=23=\dim(\fu)$. In this case $M_I(z\widehat{\xi_3})$ is irreducible.
		
		All in all, when $p=3$, $ M_I(z\widehat{\xi_3})$ is reducible if and only if
		\begin{align*}
			\begin{split}
				z&\in \{-1+\mathbb{Z}_{\geq 0}\}
				\\&\cup \{-\frac{3}{2}+\mathbb{Z}_{\geq 0}\} \\&\cup \{-\frac{2}{3}+\mathbb{Z}_{\geq 0}\} \cup \{-\frac{4}{3}+\mathbb{Z}_{\geq 0}\} \\&\cup \{-\frac{3}{4}+\mathbb{Z}_{\geq 0}\} \cup \{-\frac{5}{4}+\mathbb{Z}_{\geq 0}\} \\&\cup \{-\frac{4}{5}+\mathbb{Z}_{\geq 0}\} \cup \{-\frac{3}{5}+\mathbb{Z}_{\geq 0}\} \cup \{-\frac{2}{5}+\mathbb{Z}_{\geq 0}\} \cup \{-\frac{6}{5}+\mathbb{Z}_{\geq 0}\} \\&\cup \{-\frac{5}{6}+\mathbb{Z}_{\geq 0}\} \cup \{-\frac{7}{6}+\mathbb{Z}_{\geq 0}\} \\&\cup \{-\frac{6}{7}+\mathbb{Z}_{\geq 0}\} \cup \{-\frac{5}{7}+\mathbb{Z}_{\geq 0}\} \cup \{-\frac{4}{7}+\mathbb{Z}_{\geq 0}\} \cup \{-\frac{10}{7}+\mathbb{Z}_{\geq 0}\} \cup \{-\frac{9}{7}+\mathbb{Z}_{\geq 0}\} \cup \{-\frac{8}{7}+\mathbb{Z}_{\geq 0}\}\\&=\bigcup\limits_{i=4}^{5}\{-\frac{i+1}{i}+\frac{1}{i}\mathbb{Z}_{\geq 0}\}\cup\bigcup\limits_{i=6}^{7}\{-\frac{i+3}{i}+\frac{1}{i}\mathbb{Z}_{\geq 0}\}. 
			\end{split}
		\end{align*}
	\end{enumerate}
	
	{\bf Step 4.} Let $\Delta^+(\fl)=\{(0,1,-1,0)\}=\{\alpha_4\}$. When $M_I(\lambda)$ is of scalar type, we know that $\lambda=z(\xi_1+\xi_2+\xi_3)=z(\rho-\xi_4)$ for some $z\in\mathbb{R}$, where $\xi_4$ is the fundamental weight of simple root $\alpha_4$ and $\xi_4=(1,1,0,0)$.
	In \cite{Knapp} we know that $\rho=(\frac{11}{2},\frac{5}{2},\frac{3}{2},\frac{1}{2})$ and $\lambda+\rho=(\frac{11}{2}+\frac{9}{2}z,\frac{5}{2}+\frac{3}{2}z,\frac{3}{2}+\frac{3}{2}z,\frac{1}{2}+\frac{1}{2}z)$.
	\begin{enumerate}
		\item If $\lambda$ is integral, then we have the following.
		\begin{enumerate}
			\item When $z=-1$, then $\lambda+\rho=(1,1,0,0)$, we can get $w_\lambda=[3, 2, 1, 2, 3, 0, 1, 2,\\ 3, 1, 2, 0, 1, 2, 3]$. By using PyCox we can get that $\operatorname{GKdim} L(\lambda)=21<\dim(\fu)$.
			By Lemma \ref{reducible} we can obtain that $z=-1$ is a diagonal-reducible point.
			\item When $z=-2$, then $\lambda+\rho=(-\frac{7}{2},-\frac{1}{2},-\frac{3}{2},-\frac{1}{2})$, we can get $w_\lambda=[3]$. By using PyCox we can get that $\operatorname{GKdim} L(\lambda)=23=\dim(\fu)$.
			By Lemma \ref{reducible} we can obtain that $z=-2$ is an irreducible point.
		\end{enumerate}
		\item If $\lambda$ is half-integral, we can easily check that $\Delta_{[\lambda]}=\Delta_1\bigcup\Delta_2$, where the simple system of $\Delta_1$ is $\Pi_1=\{\beta_1=(0,0,1,1),\beta_2=(0,1,-1,0),\beta_3=(\frac{1}{2},-\frac{1}{2},\frac{1}{2},-\frac{1}{2})\}$, and the simple system of $\Delta_2$ is $\Pi_2=\{\gamma_1=(\frac{1}{2},-\frac{1}{2},-\frac{1}{2},\frac{1}{2})\}$. The Dynkin diagram of $\Delta_{[\lambda]}$ is:
		
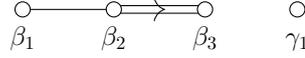
\begin{Large}
			\begin{center}
				$\dynkin[labels={\beta_1,\beta_2,\beta_3}]B3 \ \ \ \ \dynkin[labels={\gamma_1}]A1$
			\end{center}
		\end{Large}
		and with the comparison with Dynkin diagrams of classical Lie algebra, we can easily get the isomorphism $\phi:\Delta_{[\lambda]}\to B_3\times \widetilde{ A_1}$. Then $\phi(\lambda+\rho)|_{B_3}=(4z+5,2z+3,2z+2)$, $\phi(\lambda+\rho)|_{\widetilde{ A_1}}=(z+1,-z-1)$. By using RS algorithm, we can easily get that:
		\begin{enumerate}
			\item When $z=-\frac{1}{2}$, $a(w_{\lambda})=10$, and $\operatorname{GKdim} L(\lambda)=14<\dim(\fu)$. By Lemma \ref{reducible}, we can obtain that this point is an reducible point.
			\item When $z=-\frac{3}{2}$, $a(w_{\lambda})=1$, and $\operatorname{GKdim} L(\lambda)=23=\dim(\fu)$. By Lemma \ref{reducible}, we can obtain that this point is an irreducible point.
		\end{enumerate} 
		\item If $\lambda$ is one-third-integral, we can easily check that the simple system of $\Delta_{[\lambda]}$ is $\Pi=\{\beta_1=(0,0,1,0),\beta_2=(1,-1,0,0),\beta_3=(0,1,-1,0)\}$. The Dynkin diagram of $\Delta_{[\lambda]}$ is:
		
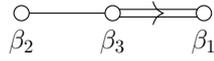
\begin{Large}
			\begin{center}
				$\dynkin[labels={\beta_2,\beta_3,\beta_1}]B3$
			\end{center}
		\end{Large}
		and with the comparison with Dynkin diagrams of classical Lie algebra, we can easily get the isomorphism $\phi:\Delta_{[\lambda]}\to B_3$. Then $\phi(\lambda+\rho)|_{B_3}=(\frac{9}{2}z+\frac{11}{2},\frac{3}{2}z+\frac{5}{2},\frac{3}{2}z+\frac{3}{2})$. By using RS algorithm, we can easily get that:
		\begin{enumerate}
			\item When $z\in\{-\frac{2}{3},-\frac{1}{3}\}$, $a(w_{\lambda})=9$, and $\operatorname{GKdim} L(\lambda)=15<\dim(\fu)$. By Lemma \ref{reducible}, we can obtain that these points are reducible points.
			\item When $z\in\{-\frac{5}{3},-\frac{4}{3}\}$, $a(w_{\lambda})=1$, and $\operatorname{GKdim} L(\lambda)=23=\dim(\fu)$. By Lemma \ref{reducible}, we can obtain that these points are irreducible points.
		\end{enumerate} 
		\item If $\lambda$ is one-fourth-integral, we can easily check that $\Delta_{[\lambda]}=\Delta_1\bigcup\Delta_2$, where the simple system of $\Delta_1$ is $\Pi_1=\{\beta_1=(0,1,-1,0),\beta_2=(\frac{1}{2},-\frac{1}{2},\frac{1}{2},-\frac{1}{2})\}$, and the simple system of $\Delta_2$ is $\Pi_2=\{\gamma_1=(\frac{1}{2},\frac{1}{2},\frac{1}{2},\frac{1}{2})\}$. The Dynkin diagram of $\Delta_{[\lambda]}$ is:
		
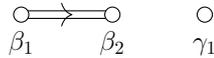
\begin{Large}
			\begin{center}
				$\dynkin[labels={\beta_1,\beta_2}]B2 \ \ \ \ \dynkin[labels={\gamma_1}]A1$
			\end{center}
		\end{Large}
		and with the comparison with Dynkin diagrams of classical Lie algebra, we can easily get the isomorphism $\phi:\Delta_{[\lambda]}\to B_2\times \widetilde{ A_1}$. Then $\phi(\lambda+\rho)|_{B_2}=(2z+3,2z+2)$, $\phi(\lambda+\rho)|_{\widetilde{ A_1}}=(4z+5,-4z-5)$. By using RS algorithm, we can easily get that:
		\begin{enumerate}
			\item When $z\in\{-\frac{3}{4},-\frac{1}{4}\}$, $a(w_{\lambda})=5$, and $\operatorname{GKdim} L(\lambda)=19<\dim(\fu)$. By Lemma \ref{reducible}, we can obtain that these points are reducible points.
			\item When $z\in\{-\frac{7}{4},-\frac{5}{4}\}$, $a(w_{\lambda})=1$, and $\operatorname{GKdim} L(\lambda)=23=\dim(\fu)$. By Lemma \ref{reducible}, we can obtain that these points are irreducible points.
		\end{enumerate} 
		\item If $\lambda$ is one-fifth-integral, we can easily check that the simple system of $\Delta_{[\lambda]}$ is $\Pi=\{\beta_1=(0,1,-1,0),\beta_2=(\frac{1}{2},-\frac{1}{2},\frac{1}{2},\frac{1}{2})\}$. The Dynkin diagram of $\Delta_{[\lambda]}$ is:
		
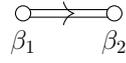
\begin{Large}
			\begin{center}
				$\dynkin[labels={\beta_1,\beta_2}]B2$
			\end{center}
		\end{Large}
		and with the comparison with Dynkin diagrams of classical Lie algebra, we can easily get the isomorphism $\phi:\Delta_{[\lambda]}\to B_2$. Then $\phi(\lambda+\rho)|_{B_2}=(\frac{5}{2}z+\frac{7}{2},\frac{5}{2}z+\frac{5}{2})$. By using RS algorithm, we can easily get that:
		\begin{enumerate}
			\item When $z\in\{-\frac{4}{5},-\frac{3}{5},-\frac{2}{5},-\frac{1}{5}\}$, $a(w_{\lambda})=4$, and $\operatorname{GKdim} L(\lambda)=20<\dim(\fu)$. By Lemma \ref{reducible}, we can obtain that these points are reducible points.
			\item When $z\in\{-\frac{9}{5},-\frac{8}{5},-\frac{7}{5},-\frac{6}{5}\}$, $a(w_{\lambda})=1$, and $\operatorname{GKdim} L(\lambda)=23=\dim(\fu)$. By Lemma \ref{reducible}, we can obtain that these points are irreducible points.
		\end{enumerate} 
		\item If $\lambda$ is one-sixth-integral, we can easily check that the simple system of $\Delta_{[\lambda]}$ is $\Pi=\{\beta_1=(1,0,1,0),\beta_2=(0,1,-1,0)\}$. The Dynkin diagram of $\Delta_{[\lambda]}$ is:
		
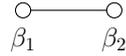
\begin{Large}
			\begin{center}
				$\dynkin[labels={\beta_1,\beta_2}]A2$
			\end{center}
		\end{Large}
		and with the comparison with Dynkin diagrams of classical Lie algebra, we can easily get the isomorphism $\phi:\Delta_{[\lambda]}\to A_2$. Then $\phi(\lambda+\rho)|_{A_2}=(4z+5,-2z-2,-2z-3)$. By using RS algorithm, we can easily get that:
		\begin{enumerate}
			\item When $z\in\{-\frac{5}{6},-\frac{1}{6}\}$, $a(w_{\lambda})=3$, and $\operatorname{GKdim} L(\lambda)=21<\dim(\fu)$. By Lemma \ref{reducible}, we can obtain that these points are reducible points.
			\item When $z\in\{-\frac{11}{6},-\frac{7}{6}\}$, $a(w_{\lambda})=1$, and $\operatorname{GKdim} L(\lambda)=23=\dim(\fu)$. By Lemma \ref{reducible}, we can obtain that these points are irreducible points.
		\end{enumerate} 
		\item If $\lambda$ is one-seventh-integral, we can easily check that $\Delta_{[\lambda]}=\Delta_1\bigcup\Delta_2$, where the simple system of $\Delta_1$ is $\Pi_1=\{\beta_1=(0,1,-1,0)\}$, and the simple system of $\Delta_2$ is $\Pi_2=\{\gamma_1=(\frac{1}{2},\frac{1}{2},\frac{1}{2},-\frac{1}{2})\}$. The Dynkin diagram of $\Delta_{[\lambda]}$ is:
		
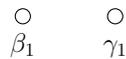
\begin{Large}
			\begin{center}
				$\dynkin[labels={\beta_1}]A1 \ \ \ \ \dynkin[labels={\gamma_1}]A1$
			\end{center}
		\end{Large}
		and with the comparison with Dynkin diagrams of classical Lie algebra, we can easily get the isomorphism $\phi:\Delta_{[\lambda]}\to A_1\times \widetilde{ A_1}$. Then $\phi(\lambda+\rho)|_{ A_1}=(\frac{1}{2},-\frac{1}{2})$, $\phi(\lambda+\rho)|_{\widetilde{ A_1}}=(\frac{7}{2}z+\frac{9}{2},-\frac{7}{2}z-\frac{9}{2})$. By using RS algorithm, we can easily get that:
		\begin{enumerate}
			\item When $z\in\{-\frac{6}{7},-\frac{5}{7},-\frac{4}{7},-\frac{3}{7},-\frac{2}{7},-\frac{8}{7}\}$, $a(w_{\lambda})=2$, and $\operatorname{GKdim} L(\lambda)=22<\dim(\fu)$. By Lemma \ref{reducible}, we can obtain that these points are reducible points.
			\item When $z\in\{-\frac{13}{7},-\frac{12}{7},-\frac{11}{7},-\frac{10}{7},-\frac{9}{7},-\frac{15}{7}\}$, $a(w_{\lambda})=1$, and $\operatorname{GKdim} L(\lambda)=23=\dim(\fu)$. By Lemma \ref{reducible}, we can obtain that these points are irreducible points.
		\end{enumerate} 
		\item If $\lambda$ is one-eighth-integral, we can easily check that $\Delta_{[\lambda]}=\Delta_1\bigcup\Delta_2$, where the simple system of $\Delta_1$ is $\Pi_1=\{\beta_1=(0,1,-1,0)\}$, and the simple system of $\Delta_2$ is $\Pi_2=\{\gamma_1=(\frac{1}{2},\frac{1}{2},\frac{1}{2},\frac{1}{2})\}$. The Dynkin diagram of $\Delta_{[\lambda]}$ is:
		\begin{Large}
			\begin{center}
				$\dynkin[labels={\beta_1}]A1 \ \ \ \ \dynkin[labels={\gamma_1}]A1$
			\end{center}
		\end{Large}
		and with the comparison with Dynkin diagrams of classical Lie algebra, we can easily get the isomorphism $\phi:\Delta_{[\lambda]}\to A_1\times \widetilde{ A_1}$. Then $\phi(\lambda+\rho)|_{ A_1}=(\frac{1}{2},-\frac{1}{2})$, $\phi(\lambda+\rho)|_{\widetilde{ A_1}}=(4z+5,-4z-5)$. By using RS algorithm, we can easily get that:
		\begin{enumerate}
			\item When $z\in\{-\frac{7}{8},-\frac{5}{8},-\frac{3}{8},-\frac{9}{8}\}$, $a(w_{\lambda})=2$, and $\operatorname{GKdim} L(\lambda)=22<\dim(\fu)$. By Lemma \ref{reducible}, we can obtain that these points are reducible points.
			\item When $z\in\{-\frac{15}{8},-\frac{13}{8},-\frac{11}{8},-\frac{17}{8}\}$, $a(w_{\lambda})=1$, and $\operatorname{GKdim} L(\lambda)=23=\dim(\fu)$. By Lemma \ref{reducible}, we can obtain that these points are irreducible points.
		\end{enumerate} 
		\item If $\lambda$ is one-ninth-integral, we can easily check that $\Delta_{[\lambda]}=\Delta_1\bigcup\Delta_2$, where the simple system of $\Delta_1$ is $\Pi_1=\{\beta_1=(1,0,0,0)\}$, and the simple system of $\Delta_2$ is $\Pi_2=\{\gamma_1=(0,1,-1,0)\}$. The Dynkin diagram of $\Delta_{[\lambda]}$ is:
		\begin{Large}
			\begin{center}
				$\dynkin[labels={\beta_1}]A1 \ \ \ \ \dynkin[labels={\gamma_1}]A1$
			\end{center}
		\end{Large}
		and with the comparison with Dynkin diagrams of classical Lie algebra, we can easily get the isomorphism $\phi:\Delta_{[\lambda]}\to \widetilde{ A_1}\times A_1$. Then $\phi(\lambda+\rho)|_{\widetilde{ A_1}}=(\frac{9}{2}z+\frac{11}{2},-\frac{9}{2}z-\frac{11}{2})$, $\phi(\lambda+\rho)|_{ A_1}=(\frac{1}{2},-\frac{1}{2})$. By using RS algorithm, we can easily get that:
		\begin{enumerate}
			\item When $z\in\{-\frac{8}{9},-\frac{7}{9},-\frac{5}{9},-\frac{4}{9},-\frac{2}{9},-\frac{10}{9}\}$, $a(w_{\lambda})=2$, and $\operatorname{GKdim} L(\lambda)=22<\dim(\fu)$. By Lemma \ref{reducible}, we can obtain that these points are reducible points.
			\item When $z\in\{-\frac{17}{9},-\frac{16}{9},-\frac{14}{9},-\frac{13}{9},-\frac{11}{9},-\frac{19}{9}\}$, $a(w_{\lambda})=1$, and $\operatorname{GKdim} L(\lambda)=23=\dim(\fu)$. By Lemma \ref{reducible}, we can obtain that these points are irreducible points.
		\end{enumerate} 
		\item If $\lambda$ is not belonging to any of the types mentioned above, one can easily check that $\Delta_{[\lambda]}\simeq  A_1$. The simple system of $\Delta_{[\lambda]}$ is $\Pi=\{\beta_1=(0,1,-1,0)\}$. We can get $a(w_{\lambda})=1$ by using RS algorithm, thus $\operatorname{GKdim} L(\lambda)=23=\dim(\fu)$. In this case $M_I(z\widehat{\xi_4})$ is irreducible.
		
		All in all, when $p=4$, $ M_I(z\widehat{\xi_4})$ is reducible if and only if
		\begin{align*}
			\begin{split}
				z&\in \{-1+\mathbb{Z}_{\geq 0}\}
				\\&\cup \{-\frac{1}{2}+\mathbb{Z}_{\geq 0}\} \\&\cup \{-\frac{2}{3}+\mathbb{Z}_{\geq 0}\} \cup \{-\frac{1}{3}+\mathbb{Z}_{\geq 0}\} \\&\cup \{-\frac{3}{4}+\mathbb{Z}_{\geq 0}\} \cup \{-\frac{1}{4}+\mathbb{Z}_{\geq 0}\} \\&\cup \{-\frac{4}{5}+\mathbb{Z}_{\geq 0}\} \cup \{-\frac{3}{5}+\mathbb{Z}_{\geq 0}\} \cup \{-\frac{2}{5}+\mathbb{Z}_{\geq 0}\} \cup \{-\frac{1}{5}+\mathbb{Z}_{\geq 0}\} \\&\cup \{-\frac{5}{6}+\mathbb{Z}_{\geq 0}\} \cup \{-\frac{1}{6}+\mathbb{Z}_{\geq 0}\} \\&\cup \{-\frac{6}{7}+\mathbb{Z}_{\geq 0}\} \cup \{-\frac{5}{7}+\mathbb{Z}_{\geq 0}\} \cup \{-\frac{4}{7}+\mathbb{Z}_{\geq 0}\} \cup \{-\frac{3}{7}+\mathbb{Z}_{\geq 0}\} \cup \{-\frac{2}{7}+\mathbb{Z}_{\geq 0}\} \cup \{-\frac{8}{7}+\mathbb{Z}_{\geq 0}\} \\&\cup \{-\frac{7}{8}+\mathbb{Z}_{\geq 0}\} \cup \{-\frac{5}{8}+\mathbb{Z}_{\geq 0}\} \cup \{-\frac{3}{8}+\mathbb{Z}_{\geq 0}\} \cup \{-\frac{9}{8}+\mathbb{Z}_{\geq 0}\} \\&\cup \{-\frac{8}{9}+\mathbb{Z}_{\geq 0}\} \cup \{-\frac{7}{9}+\mathbb{Z}_{\geq 0}\} \cup \{-\frac{5}{9}+\mathbb{Z}_{\geq 0}\} \cup \{-\frac{4}{9}+\mathbb{Z}_{\geq 0}\} \cup \{-\frac{2}{9}+\mathbb{Z}_{\geq 0}\} \cup \{-\frac{10}{9}+\mathbb{Z}_{\geq 0}\}\\&=\bigcup\limits_{i=5}^{6}\{-1+\frac{1}{i}\mathbb{Z}_{\geq 0}\}\cup\bigcup\limits_{i=7}^{9}\{-\frac{i+1}{i}+\frac{1}{i}\mathbb{Z}_{\geq 0}\}. 
			\end{split}
		\end{align*}
		
	\end{enumerate}
	So far, we have completed the proof of all cases of  type $F_4$.
\end{proof}

			\section{Reducibility of scalar generalized Verma modules for type $E$}\label{type_e}

Next we give the reducibility of scalar generalized Verma modules associated to minimal parabolic subalgebra for type $E_6$, $E_7$ and $E_8$.

\begin{Thm}\label{reducible_e6}
	Let $\fg$ is of type $E_6$, then $M_I(z\widehat{\xi_p})$ is reducible if and only if
	\begin{enumerate}
		\item  $p=1$,  $z\in\bigcup\limits_{i=6}^{10}\{-1+\frac{1}{i}\mathbb{Z}_{\geq 0}\}$; or
		
		\item  $p=2$,  $z\in \bigcup\limits_{i=5}^{9}\{-1+\frac{1}{i}\mathbb{Z}_{\geq 0}\}$; or
		
		\item  $p=3$,  $z\in \bigcup\limits_{i=5}^{7}\{-1+\frac{1}{i}\mathbb{Z}_{\geq 0}\}\cup\bigcup\limits_{i=8}^{9}\{-\frac{i+1}{i}+\frac{1}{i}\mathbb{Z}_{\geq 0}\}$; or
		
		\item  $p=4$,  $z\in \bigcup\limits_{i=4}^{8}\{-\frac{i+1}{i}+\frac{1}{i}\mathbb{Z}_{\geq 0}\}$; or
		
		\item  $p=5$, $z\in \bigcup\limits_{i=5}^{7}\{-1+\frac{1}{i}\mathbb{Z}_{\geq 0}\}\cup\bigcup\limits_{i=8}^{9}\{-\frac{i+1}{i}+\frac{1}{i}\mathbb{Z}_{\geq 0}\}$; or
		
		\item  $p=6$,  $z\in \bigcup\limits_{i=6}^{10}\{-1+\frac{1}{i}\mathbb{Z}_{\geq 0}\}$.
		
	\end{enumerate}
\end{Thm}
\begin{proof}
	Let $\fg$ be of type $E_6$ and $\Delta^+(\fl)=\{\alpha_p\}$. 
	
			


	
	{\bf Step 1.} Let $\Delta^+(\fl)=\{(\frac{1}{2},-\frac{1}{2},-\frac{1}{2},-\frac{1}{2},-\frac{1}{2},-\frac{1}{2},-\frac{1}{2},\frac{1}{2})\}=\{\alpha_1\}$. When $M_I(\lambda)$ is of scalar type, we know that $\lambda=z(\xi_2+\xi_3+\xi_4+\xi_5+\xi_6)=z(\rho-\xi_1)$ for some $z\in\mathbb{R}$, where $\xi_1$ is the fundamental weight of simple root $\alpha_1$ and $\xi_1=(0,0,0,0,0,-\frac{2}{3},-\frac{2}{3},\frac{2}{3})$.
	In \cite{Knapp} we know that $\rho=(0,1,2,3,4,-4,-4,4)$ and $\lambda+\rho=(0,1+z,2+2z,3+3z,4+4z,-4-\frac{10}{3}z,-4-\frac{10}{3}z,4+\frac{10}{3}z)$.
	\begin{enumerate}
		\item If $\lambda$ is integral, then we have the following.
		\begin{enumerate}
			\item When $z=-1$, then $\lambda+\rho=(0,0,0,0,0,-\frac{2}{3},-\frac{2}{3},\frac{2}{3})$, we can get $w_\lambda=[5, 4, 3, 2, 0, 1, 3, 2, 4, 3, 1, 5, 4, 3, 2, 0]$. By using PyCox we can get that $\operatorname{GKdim} L(\lambda)=34<\dim(\fu)$.
			By Lemma \ref{reducible} we can obtain that $z=-1$ is a diagonal-reducible point.
			\item When $z=-2$, then $\lambda+\rho=(0,-1,-2,-3,-4,\frac{8}{3},\frac{8}{3},-\frac{8}{3})$, we can get $w_\lambda=[0]$. By using PyCox we can get that $\operatorname{GKdim} L(\lambda)=35=\dim(\fu)$.
			By Lemma \ref{reducible} we can obtain that $z=-2$ is an irreducible point.
		\end{enumerate}
		\item If $\lambda$ is half-integral, we can easily check that $\Delta_{[\lambda]}=\Delta_1\bigcup\Delta_2$, where the simple system of $\Delta_1$ is $\Pi_1=\{\beta_1=(1,0,1,0,0,0,0,0),\beta_2=(0,1,0,1,0,0,0,0),\\ \beta_3=(-1,0,1,0,0,0,0,0),\beta_4=(0,0,-1,0,1,0,0,0),\beta_5=(\frac{1}{2},-\frac{1}{2},-\frac{1}{2},-\frac{1}{2},-\frac{1}{2},
		\\
		-\frac{1}{2},-\frac{1}{2},\frac{1}{2})\}$, and the simple system of $\Delta_2$ is $\Pi_2=\{\gamma_1=(0,-1,0,1,0,0,0,0)\}$. The Dynkin diagram of $\Delta_{[\lambda]}$ is:
		
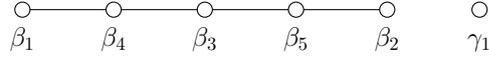
\begin{Large}
			\begin{center}
				$\dynkin[labels={\beta_1,\beta_4,\beta_3,\beta_5,\beta_2}]A5 \ \ \ \ \dynkin[labels={\gamma_1}]A1$
			\end{center}
		\end{Large}
		and with the comparison with Dynkin diagrams of classical Lie algebra, we can easily get the isomorphism $\phi:\Delta_{[\lambda]}\to A_5\times A_1$. Then $\phi(\lambda+\rho)|_{A_5}=(\frac{14}{3}z+5,\frac{8}{3}z+3,\frac{2}{3}z+1,-\frac{4}{3}z-1,-\frac{4}{3}z-2,-\frac{16}{3}z-6)$, $\phi(\lambda+\rho)|_{A_1}=(z+1,-z-1)$. By using RS algorithm, we can easily get that:
		\begin{enumerate}
			\item When $z=-\frac{1}{2}$, $a(w_{\lambda})=16$, and $\operatorname{GKdim} L(\lambda)=20<\dim(\fu)$. By Lemma \ref{reducible}, we can obtain that this point is an reducible point.
			\item When $z=-\frac{3}{2}$, $a(w_{\lambda})=1$, and $\operatorname{GKdim} L(\lambda)=35=\dim(\fu)$. By Lemma \ref{reducible}, we can obtain that this point is an irreducible point.
		\end{enumerate} 
		\item If $\lambda$ is one-third-integral, we can easily check that $\Delta_{[\lambda]}=\Delta_1\bigcup\Delta_2$, where the simple system of $\Delta_1$ is $\Pi_1=\{\beta_1=(1,0,0,1,0,0,0,0)\}$, and the simple system of $\Delta_2$ is $\Pi_2=\{\gamma_1=(0,1,1,0,0,0,0,0),\gamma_2=(-1,0,0,1,0,0,0,0),\gamma_3=(0,-1,0,0,1,0,0,0),\gamma_4=(\frac{1}{2},-\frac{1}{2},-\frac{1}{2},-\frac{1}{2},-\frac{1}{2},-\frac{1}{2},-\frac{1}{2},\frac{1}{2})\}$. The Dynkin diagram of $\Delta_{[\lambda]}$ is:
		
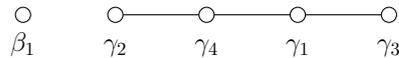
\begin{Large}
			\begin{center}
				$\dynkin[labels={\beta_1}]A1 \ \ \ \ \dynkin[labels={\gamma_2,\gamma_4,\gamma_1,\gamma_3}]A4$
			\end{center}
		\end{Large}
		and with the comparison with Dynkin diagrams of classical Lie algebra, we can easily get the isomorphism $\phi:\Delta_{[\lambda]}\to A_1\times A_4$. Then $\phi(\lambda+\rho)|_{A_1}=(\frac{3}{2}z+\frac{3}{2},-\frac{3}{2}z-\frac{3}{2})$, $\phi(\lambda+\rho)|_{A_4}=(\frac{21}{5}z+\frac{24}{5},\frac{6}{5}z+\frac{9}{5},\frac{6}{5}z+\frac{4}{5},-\frac{9}{5}z-\frac{11}{5},-\frac{24}{5}z-\frac{26}{5})$. By using RS algorithm, we can easily get that:
		\begin{enumerate}
			\item When $z\in\{-\frac{2}{3},-\frac{1}{3}\}$, $a(w_{\lambda})=11$, and $\operatorname{GKdim} L(\lambda)=25<\dim(\fu)$. By Lemma \ref{reducible}, we can obtain that these points are reducible points.
			\item When $z\in\{-\frac{5}{3},-\frac{4}{3}\}$, $a(w_{\lambda})=1$, and $\operatorname{GKdim} L(\lambda)=35=\dim(\fu)$. By Lemma \ref{reducible}, we can obtain that these points are irreducible points.
		\end{enumerate} 
		\item If $\lambda$ is one-fourth-integral, we can easily check that $\Delta_{[\lambda]}=\Delta_1\bigcup\Delta_2$, where the simple system of $\Delta_1$ is $\Pi_1=\{\beta_1=(1,0,0,0,1,0,0,0)\}$, and the simple system of $\Delta_2$ is $\Pi_2=\{\gamma_1=(0,1,0,1,0,0,0,0),\gamma_2=(-1,0,0,0,1,0,0,\\
		0),\gamma_3=(\frac{1}{2},-\frac{1}{2},-\frac{1}{2},-\frac{1}{2},-\frac{1}{2},-\frac{1}{2},-\frac{1}{2},\frac{1}{2})\}$. The Dynkin diagram of $\Delta_{[\lambda]}$ is:
		
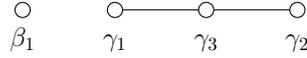
\begin{Large}
			\begin{center}
				$\dynkin[labels={\beta_1}]A1 \ \ \ \ \dynkin[labels={\gamma_1,\gamma_3,\gamma_2}]A3$
			\end{center}
		\end{Large}
		and with the comparison with Dynkin diagrams of classical Lie algebra, we can easily get the isomorphism $\phi:\Delta_{[\lambda]}\to A_1\times A_3$. Then $\phi(\lambda+\rho)|_{A_1}=(2z+2,-2z-2)$, $\phi(\lambda+\rho)|_{A_3}=(4z+\frac{9}{2},\frac{1}{2},-\frac{1}{2},-4z-\frac{9}{2})$. By using RS algorithm, we can easily get that:
		\begin{enumerate}
			\item When $z\in\{-\frac{3}{4},-\frac{1}{4}\}$, $a(w_{\lambda})=7$, and $\operatorname{GKdim} L(\lambda)=29<\dim(\fu)$. By Lemma \ref{reducible}, we can obtain that these points are reducible points.
			\item When $z\in\{-\frac{7}{4},-\frac{5}{4}\}$, $a(w_{\lambda})=1$, and $\operatorname{GKdim} L(\lambda)=35=\dim(\fu)$. By Lemma \ref{reducible}, we can obtain that these points are irreducible points.
		\end{enumerate} 
		\item If $\lambda$ is one-fifth-integral, we can easily check that the simple system of $\Delta_{[\lambda]}$ is $\Pi=\{\beta_1=(0,1,0,0,1,0,0,0),\beta_2=(0,0,1,1,0,0,0,0),\beta_3=(\frac{1}{2},-\frac{1}{2},-\frac{1}{2},-\frac{1}{2},$
        
        $-\frac{1}{2},-\frac{1}{2},-\frac{1}{2},\frac{1}{2})\}$. The Dynkin diagram of $\Delta_{[\lambda]}$ is:
		
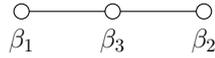
\begin{Large}
			\begin{center}
				$\dynkin[labels={\beta_1,\beta_3,\beta_2}]A3$
			\end{center}
		\end{Large}
		and with the comparison with Dynkin diagrams of classical Lie algebra, we can easily get the isomorphism $\phi:\Delta_{[\lambda]}\to A_3$. Then $\phi(\lambda+\rho)|_{A_3}=(5z+\frac{11}{2},\frac{1}{2},-\frac{1}{2},-5z-\frac{11}{2})$. By using RS algorithm, we can easily get that:
		\begin{enumerate}
			\item When $z\in\{-\frac{4}{5},-\frac{3}{5},-\frac{2}{5},-\frac{1}{5}\}$, $a(w_{\lambda})=6$, and $\operatorname{GKdim} L(\lambda)=30<\dim(\fu)$. By Lemma \ref{reducible}, we can obtain that these points are reducible points.
			\item When $z\in\{-\frac{9}{5},-\frac{8}{5},-\frac{7}{5},-\frac{6}{5}\}$, $a(w_{\lambda})=1$, and $\operatorname{GKdim} L(\lambda)=35=\dim(\fu)$. By Lemma \ref{reducible}, we can obtain that these points are irreducible points.
		\end{enumerate} 
		\item If $\lambda$ is one-sixth-integral, we can easily check that $\Delta_{[\lambda]}=\Delta_1\bigcup\Delta_2$, where the simple system of $\Delta_1$ is $\Pi_1=\{\beta_1=(0,0,1,0,1,0,0,0),\beta_2=(\frac{1}{2},-\frac{1}{2},-\frac{1}{2},-\frac{1}{2},$
        
        $-\frac{1}{2},-\frac{1}{2},-\frac{1}{2},\frac{1}{2})\}$, and the simple system of $\Delta_2$ is $\Pi_2=\{\gamma_1=(-\frac{1}{2},\frac{1}{2},\frac{1}{2},\frac{1}{2},-\frac{1}{2},$
        
        $-\frac{1}{2},-\frac{1}{2},\frac{1}{2})\}$. The Dynkin diagram of $\Delta_{[\lambda]}$ is:
		
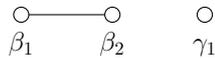
\begin{Large}
			\begin{center}
				$\dynkin[labels={\beta_1,\beta_2}]A2 \ \ \ \ \dynkin[labels={\gamma_1}]A1$
			\end{center}
		\end{Large}
		and with the comparison with Dynkin diagrams of classical Lie algebra, we can easily get the isomorphism $\phi:\Delta_{[\lambda]}\to A_2\times A_1$. Then $\phi(\lambda+\rho)|_{A_2}=(4z+\frac{13}{3},-2z-\frac{5}{3},-2z-\frac{8}{3})$, $\phi(\lambda+\rho)|_{A_1}=(3z+\frac{7}{2},-3z-\frac{7}{2})$. By using RS algorithm, we can easily get that:
		\begin{enumerate}
			\item When $z\in\{-\frac{5}{6},-\frac{1}{6}\}$, $a(w_{\lambda})=4$, and $\operatorname{GKdim} L(\lambda)=32<\dim(\fu)$. By Lemma \ref{reducible}, we can obtain that these points are reducible points.
			\item When $z\in\{-\frac{11}{6},-\frac{7}{6}\}$, $a(w_{\lambda})=1$, and $\operatorname{GKdim} L(\lambda)=35=\dim(\fu)$. By Lemma \ref{reducible}, we can obtain that these points are irreducible points.
		\end{enumerate} 
		\item If $\lambda$ is one-seventh-integral, we can easily check that $\Delta_{[\lambda]}=\Delta_1\bigcup\Delta_2$, where the simple system of $\Delta_1$ is $\Pi_1=\{\beta_1=(0,0,0,1,1,0,0,0),\beta_2=(\frac{1}{2},-\frac{1}{2},-\frac{1}{2},-\frac{1}{2},$
        
        $-\frac{1}{2},-\frac{1}{2},-\frac{1}{2},\frac{1}{2})\}$, and the simple system of $\Delta_2$ is $\Pi_2=\{\gamma_1=(-\frac{1}{2},\frac{1}{2},\frac{1}{2},-\frac{1}{2},\frac{1}{2},$
        
        $-\frac{1}{2},-\frac{1}{2},\frac{1}{2})\}$. The Dynkin diagram of $\Delta_{[\lambda]}$ is:
		
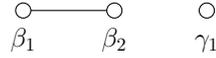
\begin{Large}
			\begin{center}
				$\dynkin[labels={\beta_1,\beta_2}]A2 \ \ \ \ \dynkin[labels={\gamma_1}]A1$
			\end{center}
		\end{Large}
		and with the comparison with Dynkin diagrams of classical Lie algebra, we can easily get the isomorphism $\phi:\Delta_{[\lambda]}\to A_2\times A_1$. Then $\phi(\lambda+\rho)|_{A_2}=(\frac{14}{3}z+5,-\frac{7}{3}z-2,-\frac{7}{3}z-3)$, $\phi(\lambda+\rho)|_{A_1}=(\frac{7}{2}z+4,-\frac{7}{2}z-4)$. By using RS algorithm, we can easily get that:
		\begin{enumerate}
			\item When $z\in\{-\frac{6}{7},-\frac{5}{7},-\frac{4}{7},-\frac{3}{7},-\frac{2}{7},-\frac{1}{7}\}$, $a(w_{\lambda})=4$, and $\operatorname{GKdim} L(\lambda)=32<\dim(\fu)$. By Lemma \ref{reducible}, we can obtain that these points are reducible points.
			\item When $z\in\{-\frac{13}{7},-\frac{12}{7},-\frac{11}{7},-\frac{10}{7},-\frac{9}{7},-\frac{8}{7}\}$, $a(w_{\lambda})=1$, and $\operatorname{GKdim} L(\lambda)=35=\dim(\fu)$. By Lemma \ref{reducible}, we can obtain that these points are irreducible points.
		\end{enumerate} 
		\item If $\lambda$ is one-eighth-integral, we can easily check that $\Delta_{[\lambda]}=\Delta_1\bigcup\Delta_2$, where the simple system of $\Delta_1$ is $\Pi_1=\{\beta_1=(\frac{1}{2},-\frac{1}{2},-\frac{1}{2},-\frac{1}{2},-\frac{1}{2},-\frac{1}{2},-\frac{1}{2},\frac{1}{2})\}$, and the simple system of $\Delta_2$ is $\Pi_2=\{\gamma_1=(-\frac{1}{2},\frac{1}{2},-\frac{1}{2},\frac{1}{2},\frac{1}{2},-\frac{1}{2},-\frac{1}{2},\frac{1}{2})\}$. The Dynkin diagram of $\Delta_{[\lambda]}$ is:
		
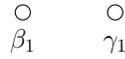
\begin{Large}
			\begin{center}
				$\dynkin[labels={\beta_1}]A1 \ \ \ \ \dynkin[labels={\gamma_1}]A1$
			\end{center}
		\end{Large}
		and with the comparison with Dynkin diagrams of classical Lie algebra, we can easily get the isomorphism $\phi:\Delta_{[\lambda]}\to A_1\times A_1$. Then $\phi(\lambda+\rho)|_{A_1}=(\frac{1}{2},-\frac{1}{2})$, $\phi(\lambda+\rho)|_{A_1}=(4z+\frac{9}{2},-4z-\frac{9}{2})$. By using RS algorithm, we can easily get that:
		\begin{enumerate}
			\item When $z\in\{-\frac{7}{8},-\frac{5}{8},-\frac{3}{8},-\frac{1}{8}\}$, $a(w_{\lambda})=2$, and $\operatorname{GKdim} L(\lambda)=34<\dim(\fu)$. By Lemma \ref{reducible}, we can obtain that these points are reducible points.
			\item When $z\in\{-\frac{15}{8},-\frac{13}{8},-\frac{11}{8},-\frac{9}{8}\}$, $a(w_{\lambda})=1$, and $\operatorname{GKdim} L(\lambda)=35=\dim(\fu)$. By Lemma \ref{reducible}, we can obtain that these points are irreducible points.
		\end{enumerate} 
		\item If $\lambda$ is one-ninth-integral, we can easily check that $\Delta_{[\lambda]}=\Delta_1\bigcup\Delta_2$, where the simple system of $\Delta_1$ is $\Pi_1=\{\beta_1=(\frac{1}{2},-\frac{1}{2},-\frac{1}{2},-\frac{1}{2},-\frac{1}{2},-\frac{1}{2},-\frac{1}{2},\frac{1}{2})\}$, and the simple system of $\Delta_2$ is $\Pi_2=\{\gamma_1=(-\frac{1}{2},-\frac{1}{2},\frac{1}{2},\frac{1}{2},\frac{1}{2},-\frac{1}{2},-\frac{1}{2},\frac{1}{2})\}$. The Dynkin diagram of $\Delta_{[\lambda]}$ is:
		\begin{Large}
			\begin{center}
				$\dynkin[labels={\beta_1}]A1 \ \ \ \ \dynkin[labels={\gamma_1}]A1$
			\end{center}
		\end{Large}
		and with the comparison with Dynkin diagrams of classical Lie algebra, we can easily get the isomorphism $\phi:\Delta_{[\lambda]}\to A_1\times A_1$. Then $\phi(\lambda+\rho)|_{A_1}=(\frac{1}{2},-\frac{1}{2})$, $\phi(\lambda+\rho)|_{A_1}=(\frac{9}{2}z+5,-\frac{9}{2}z-5)$. By using RS algorithm, we can easily get that:
		\begin{enumerate}
			\item When $z\in\{-\frac{8}{9},-\frac{7}{9},-\frac{5}{9},-\frac{4}{9},-\frac{2}{9},-\frac{1}{9}\}$, $a(w_{\lambda})=2$, and $\operatorname{GKdim} L(\lambda)=34<\dim(\fu)$. By Lemma \ref{reducible}, we can obtain that these points are reducible points.
			\item When $z\in\{-\frac{17}{9},-\frac{16}{9},-\frac{14}{9},-\frac{13}{9},-\frac{11}{9},-\frac{10}{9}\}$, $a(w_{\lambda})=1$, and $\operatorname{GKdim} L(\lambda)=35=\dim(\fu)$. By Lemma \ref{reducible}, we can obtain that these points are irreducible points.
		\end{enumerate} 
		\item If $\lambda$ is one-tenth-integral, we can easily check that $\Delta_{[\lambda]}=\Delta_1\bigcup\Delta_2$, where the simple system of $\Delta_1$ is $\Pi_1=\{\beta_1=(\frac{1}{2},\frac{1}{2},\frac{1}{2},\frac{1}{2},\frac{1}{2},-\frac{1}{2},-\frac{1}{2},\frac{1}{2})\}$, and the simple system of $\Delta_2$ is $\Pi_2=\{\gamma_1=(\frac{1}{2},-\frac{1}{2},-\frac{1}{2},-\frac{1}{2},-\frac{1}{2},-\frac{1}{2},-\frac{1}{2},\frac{1}{2})\}$. The Dynkin diagram of $\Delta_{[\lambda]}$ is:
		\begin{Large}
			\begin{center}
				$\dynkin[labels={\beta_1}]A1 \ \ \ \ \dynkin[labels={\gamma_1}]A1$
			\end{center}
		\end{Large}
		and with the comparison with Dynkin diagrams of classical Lie algebra, we can easily get the isomorphism $\phi:\Delta_{[\lambda]}\to A_1\times A_1$. Then $\phi(\lambda+\rho)|_{A_1}=(5z+\frac{11}{2},-5z-\frac{11}{2})$, $\phi(\lambda+\rho)|_{A_1}=(\frac{1}{2},-\frac{1}{2})$. By using RS algorithm, we can easily get that:
		\begin{enumerate}
			\item When $z\in\{-\frac{9}{10},-\frac{7}{10},-\frac{3}{10},-\frac{1}{10}\}$, $a(w_{\lambda})=2$, and $\operatorname{GKdim} L(\lambda)=34<\dim(\fu)$. By Lemma \ref{reducible}, we can obtain that these points are reducible points.
			\item When $z\in\{-\frac{19}{10},-\frac{17}{10},-\frac{13}{10},-\frac{11}{10}\}$, $a(w_{\lambda})=1$, and $\operatorname{GKdim} L(\lambda)=35=\dim(\fu)$. By Lemma \ref{reducible}, we can obtain that these points are irreducible points.
		\end{enumerate} 
		\item If $\lambda$ is not belonging to any of the types mentioned above, one can easily check that $\Delta_{[\lambda]}\simeq  A_1$. The simple system of $\Delta_{[\lambda]}$ is $\Pi=\{\beta_1=(\frac{1}{2},-\frac{1}{2},-\frac{1}{2},-\frac{1}{2},-\frac{1}{2},$
        
        $-\frac{1}{2},-\frac{1}{2},\frac{1}{2})\}$. We can get $a(w_{\lambda})=1$ by using RS algorithm, thus $\operatorname{GKdim} L(\lambda)=35=\dim(\fu)$. In this case, $M_I(z\widehat{\xi_1})$ is irreducible.
		
		All in all, when $p=1$, $ M_I(z\widehat{\xi_1})$ is reducible if and only if
		\begin{align*}
				z&\in \{-1+\mathbb{Z}_{\geq 0}\}
				\cup \{-\frac{1}{2}+\mathbb{Z}_{\geq 0}\} \\&\cup \{-\frac{2}{3}+\mathbb{Z}_{\geq 0}\} \cup \{-\frac{1}{3}+\mathbb{Z}_{\geq 0}\} \\&\cup \{-\frac{3}{4}+\mathbb{Z}_{\geq 0}\} \cup \{-\frac{1}{4}+\mathbb{Z}_{\geq 0}\} \\&\cup \{-\frac{4}{5}+\mathbb{Z}_{\geq 0}\} \cup \{-\frac{3}{5}+\mathbb{Z}_{\geq 0}\} \cup \{-\frac{2}{5}+\mathbb{Z}_{\geq 0}\} \cup \{-\frac{1}{5}+\mathbb{Z}_{\geq 0}\} \\&\cup \{-\frac{5}{6}+\mathbb{Z}_{\geq 0}\} \cup \{-\frac{1}{6}+\mathbb{Z}_{\geq 0}\} \\&\cup \{-\frac{6}{7}+\mathbb{Z}_{\geq 0}\} \cup \{-\frac{5}{7}+\mathbb{Z}_{\geq 0}\} \cup \{-\frac{4}{7}+\mathbb{Z}_{\geq 0}\} \cup \{-\frac{3}{7}+\mathbb{Z}_{\geq 0}\} \cup \{-\frac{2}{7}+\mathbb{Z}_{\geq 0}\} \cup \{-\frac{1}{7}+\mathbb{Z}_{\geq 0}\} \\&\cup \{-\frac{7}{8}+\mathbb{Z}_{\geq 0}\} \cup \{-\frac{5}{8}+\mathbb{Z}_{\geq 0}\} \cup \{-\frac{3}{8}+\mathbb{Z}_{\geq 0}\} \cup \{-\frac{1}{8}+\mathbb{Z}_{\geq 0}\} \\&\cup \{-\frac{8}{9}+\mathbb{Z}_{\geq 0}\} \cup \{-\frac{7}{9}+\mathbb{Z}_{\geq 0}\} \cup \{-\frac{5}{9}+\mathbb{Z}_{\geq 0}\} \cup \{-\frac{4}{9}+\mathbb{Z}_{\geq 0}\} \cup \{-\frac{2}{9}+\mathbb{Z}_{\geq 0}\} \cup \{-\frac{1}{9}+\mathbb{Z}_{\geq 0}\} \\&\cup \{-\frac{9}{10}+\mathbb{Z}_{\geq 0}\} \cup \{-\frac{7}{10}+\mathbb{Z}_{\geq 0}\} \cup \{-\frac{3}{10}+\mathbb{Z}_{\geq 0}\} \cup \{-\frac{1}{10}+\mathbb{Z}_{\geq 0}\}\\&=\bigcup\limits_{i=6}^{10}\{-1+\frac{1}{i}\mathbb{Z}_{\geq 0}\}.
		\end{align*}

	\end{enumerate}
	 The computations for $E_6$ are tedious, so we only present the detailed procedure for the first case of $E_6$, and omit subsequent derivations for brevity.

\end{proof}

\begin{Thm}\label{reducible_e7}
	Let $\fg$ is of type $E_7$, then $M_I(z\widehat{\xi_p})$ is reducible if and only if
	
	\begin{enumerate}
		\item  $p=1$, then $z\in\bigcup\limits_{i=8}^{15}\{-1+\frac{1}{i}\mathbb{Z}_{\geq 0}\} $; or
		\item  $p=2$, then $z \in \bigcup\limits_{i=8}^{12}\{-1+\frac{1}{i}\mathbb{Z}_{\geq 0}\}\cup\bigcup\limits_{i=13}^{15}\{-\frac{i+1}{i}+\frac{1}{i}\mathbb{Z}_{\geq 0}\}$; or
		\item  $p=3$, then $z\in \bigcup\limits_{i=8}^{13}\{-\frac{i+1}{i}+\frac{1}{i}\mathbb{Z}_{\geq 0}\}\cup\{-\frac{8}{7}+\frac{1}{14}\mathbb{Z}_{\geq 0}\}$; or
		\item  $p=4$, then $z\in \bigcup\limits_{i=7}^{8}\{-\frac{i+1}{i}+\frac{1}{i}\mathbb{Z}_{\geq 0}\}\cup\bigcup\limits_{i=9}^{11}\{-\frac{i+2}{i}+\frac{1}{i}\mathbb{Z}_{\geq 0}\}\cup\bigcup\limits_{i=12}^{13}\{-\frac{i+3}{i}+\frac{1}{i}\mathbb{Z}_{\geq 0}\}$; or
		\item  $p=5$, then $z\in \bigcup\limits_{i=8}^{11}\{-\frac{i+1}{i}+\frac{1}{i}\mathbb{Z}_{\geq 0}\}\cup\bigcup\limits_{i=12}^{14}\{-\frac{i+2                               }{i}+\frac{1}{i}\mathbb{Z}_{\geq 0}\}$; or
		\item  $p=6$, then $z\in \bigcup\limits_{i=8}^{10}\{-1+\frac{1}{i}\mathbb{Z}_{\geq 0}\}\cup\bigcup\limits_{i=11}^{15}\{-\frac{i+1}{i}+\frac{1}{i}\mathbb{Z}_{\geq 0}\}$; or
		\item  $p=7$, then $z\in \bigcup\limits_{i=8}^{16}\{-1+\frac{1}{i}\mathbb{Z}_{\geq 0}\}$.
	\end{enumerate}
\end{Thm}

\begin{proof}
		 The computations for $E_7$ are tedious and the same methodology has employed in the previous cases, so we omitted detailed procedures here. 
\end{proof}

\begin{Thm}\label{reducible_e8}
	Let $\fg$ is of type $E_8$, then $M_I(z\widehat{\xi_p})$ is reducible if and only if
	\begin{enumerate}
		\item  $p=1$, then $z\in \bigcup\limits_{i=14}^{20}\{-1+\frac{1}{i}\mathbb{Z}_{\geq 0}\}\cup\bigcup\limits_{i=21}^{27}\{-\frac{i+1}{i}+\frac{1}{i}\mathbb{Z}_{\geq 0}\}$; or
		\item  $p=2$, then $z\in 
		\bigcup\limits_{i=14}^{21}\{-\frac{i+1}{i}+\frac{1}{i}\mathbb{Z}_{\geq 0}\}\cup\bigcup\limits_{i=22}^{26}\{-\frac{i+2}{i}+\frac{1}{i}\mathbb{Z}_{\geq 0}\}$; or
		\item  $p=3$, then $z\in 
		\{-\frac{14}{13}+\frac{1}{13}\mathbb{Z}_{\geq 0}\}\cup\bigcup\limits_{i=14}^{20}\{-\frac{i+2}{i}+\frac{1}{i}\mathbb{Z}_{\geq 0}\}\cup\bigcup\limits_{i=21}^{25}\{-\frac{i+3}{i}+\frac{1}{i}\mathbb{Z}_{\geq 0}\}$; or
		\item  $p=4$, then $z\in 
		\{-\frac{7}{6}+\frac{1}{12}\mathbb{Z}_{\geq 0}\}\cup\bigcup\limits_{i=13}^{15}\{-\frac{i+3}{i}+\frac{1}{i}\mathbb{Z}_{\geq 0}\}\cup\bigcup\limits_{i=16}^{20}\{-\frac{i+4}{i}+\frac{1}{i}\mathbb{Z}_{\geq 0}\}\cup\bigcup\limits_{i=21}^{23}\{-\frac{i+5}{i}+\frac{1}{i}\mathbb{Z}_{\geq 0}\}$; or
		\item  $p=5$, then $z\in 
		\bigcup\limits_{i=13}^{16}\{-\frac{i+2}{i}+\frac{1}{i}\mathbb{Z}_{\geq 0}\}\cup\bigcup\limits_{i=17}^{21}\{-\frac{i+3}{i}+\frac{1}{i}\mathbb{Z}_{\geq 0}\}\cup\bigcup\limits_{i=22}^{24}\{-\frac{i+4}{i}+\frac{1}{i}\mathbb{Z}_{\geq 0}\}$; or
		\item  $p=6$, then $z\in 
		\bigcup\limits_{i=13}^{16}\{-\frac{i+1}{i}+\frac{1}{i}\mathbb{Z}_{\geq 0}\}\cup\bigcup\limits_{i=17}^{23}\{-\frac{i+2}{i}+\frac{1}{i}\mathbb{Z}_{\geq 0}\}\cup\bigcup\limits_{i=24}^{25}\{-\frac{i+3}{i}+\frac{1}{i}\mathbb{Z}_{\geq 0}\}$; or
		
		\item  $p=7$, then $z\in \bigcup\limits_{i=13}^{26}\{-\frac{i+1}{i}+\frac{1}{i}\mathbb{Z}_{\geq 0}\}$; or
		
		\item  $p=8$, then $z\in \bigcup\limits_{i=14}^{27}\{-1+\frac{1}{i}\mathbb{Z}_{\geq 0}\}$.
	\end{enumerate}
\end{Thm}
\begin{proof}
	The computations for $E_8$ are tedious and the same methodology has employed in the previous cases. To avoid lengthy articles, we omit detailed procedures here. 
\end{proof}

\section{Appendix}\label{app}
In the following tables, $p$ means the case of $\lambda=z\widehat{\xi_p}$, and we give the sets of corresponding reducible points in the second column of the tables.

\makeatletter\def\@captype{table}\makeatother
\centering
\renewcommand{\arraystretch}{1.4}
\setlength\tabcolsep{5pt}
\caption{Diagonal-reducible points of $G_2$}
\label{constants}
{	\begin{tabular}{|c|c|}			
		\hline
		$p$ & diagonal-reducible points of $M_I(z\widehat{\xi_p})$  \\ 
		\hline 
		1& $-\frac{1}{2}+\frac{1}{2}\mathbb{Z}_{\geq 0},-\frac{2}{3}+\frac{1}{3}\mathbb{Z}_{\geq 0}$ \\ \hline
		2& $-\frac{3}{2}+\mathbb{Z}_{\geq 0},\mathbb{Z}_{\geq 0}$ \\ \hline
		
\end{tabular}}

%

\centering
\renewcommand{\arraystretch}{1.4}
\setlength\tabcolsep{5pt}
\caption{Diagonal-reducible points of $F_4$}
\begin{longtable*}{|c|c|}
	\hline
	$p$ & diagonal-reducible points of $M_I(z\widehat{\xi_p})$  \\ 
	\endfirsthead
\endfoot
\endlastfoot
\hline
		1& $\bigcup\limits_{i=5}^{9}\{-1+\frac{1}{i}\mathbb{Z}_{\geq 0}\}$ \\ \hline
		2& $\bigcup\limits_{i=5}^{8}\{-\frac{i+1}{i}+\frac{1}{i}\mathbb{Z}_{\geq 0}\}$ \\ \hline
		3&$\bigcup\limits_{i=4}^{5}\{-\frac{i+1}{i}+\frac{1}{i}\mathbb{Z}_{\geq 0}\},\bigcup\limits_{i=6}^{7}\{-\frac{i+3}{i}+\frac{1}{i}\mathbb{Z}_{\geq 0}\}$ \\ \hline
		4& $\bigcup\limits_{i=5}^{6}\{-1+\frac{1}{i}\mathbb{Z}_{\geq 0}\},\bigcup\limits_{i=7}^{9}\{-\frac{i+1}{i}+\frac{1}{i}\mathbb{Z}_{\geq 0}\}$\\ \hline
\end{longtable*}
%
\centering
\renewcommand{\arraystretch}{1.4}
\setlength\tabcolsep{5pt}
\caption{Diagonal-reducible points of $E_6$}
\begin{longtable*}{|c|c|}
	
	\hline
	$p$ & diagonal-reducible points of $M_I(z\widehat{\xi_p})$  \\ 
	\endfirsthead
	\endfoot
	\endlastfoot

\hline 
1& $\bigcup\limits_{i=6}^{10}\{-1+\frac{1}{i}\mathbb{Z}_{\geq 0}\}$ \\ \hline
2& $\bigcup\limits_{i=5}^{9}\{-1+\frac{1}{i}\mathbb{Z}_{\geq 0}\}$ \\ \hline
3&$\bigcup\limits_{i=5}^{7}\{-1+\frac{1}{i}\mathbb{Z}_{\geq 0}\},\bigcup\limits_{i=8}^{9}\{-\frac{i+1}{i}+\frac{1}{i}\mathbb{Z}_{\geq 0}\}$ \\ \hline
4& $\bigcup\limits_{i=4}^{8}\{-\frac{i+1}{i}+\frac{1}{i}\mathbb{Z}_{\geq 0}\}$\\ \hline
5& $\bigcup\limits_{i=5}^{7}\{-1+\frac{1}{i}\mathbb{Z}_{\geq 0}\},\bigcup\limits_{i=8}^{9}\{-\frac{i+1}{i}+\frac{1}{i}\mathbb{Z}_{\geq 0}\}$ \\ \hline
6&  $\bigcup\limits_{i=6}^{10}\{-1+\frac{1}{i}\mathbb{Z}_{\geq 0}\}$ \\ \hline
	
\end{longtable*}

%
\centering
\renewcommand{\arraystretch}{1.4}
\setlength\tabcolsep{5pt}
\caption{Diagonal-reducible points of $E_7$}
\begin{longtable*}{|c|c|}
	\hline
$p$ & diagonal-reducible points of $M_I(z\widehat{\xi_p})$  \\ 
	\endfirsthead
	\endfoot
	\endlastfoot
\hline 
1& $\bigcup\limits_{i=8}^{15}\{-1+\frac{1}{i}\mathbb{Z}_{\geq 0}\}$ \\ \hline
2& $\bigcup\limits_{i=8}^{12}\{-1+\frac{1}{i}\mathbb{Z}_{\geq 0}\},\bigcup\limits_{i=13}^{15}\{-\frac{i+1}{i}+\frac{1}{i}\mathbb{Z}_{\geq 0}\}$ \\ \hline
3&$\bigcup\limits_{i=8}^{13}\{-\frac{i+1}{i}+\frac{1}{i}\mathbb{Z}_{\geq 0}\},-\frac{8}{7}+\frac{1}{14}\mathbb{Z}_{\geq 0}$ \\ \hline
4& $\bigcup\limits_{i=7}^{8}\{-\frac{i+1}{i}+\frac{1}{i}\mathbb{Z}_{\geq 0}\},\bigcup\limits_{i=9}^{11}\{-\frac{i+2}{i}+\frac{1}{i}\mathbb{Z}_{\geq 0}\},\bigcup\limits_{i=12}^{13}\{-\frac{i+3}{i}+\frac{1}{i}\mathbb{Z}_{\geq 0}\}$\\ \hline
5& $\bigcup\limits_{i=8}^{11}\{-\frac{i+1}{i}+\frac{1}{i}\mathbb{Z}_{\geq 0}\},\bigcup\limits_{i=12}^{14}\{-\frac{i+2                               }{i}+\frac{1}{i}\mathbb{Z}_{\geq 0}\}$\\ \hline
6&  $\bigcup\limits_{i=8}^{10}\{-1+\frac{1}{i}\mathbb{Z}_{\geq 0}\},\bigcup\limits_{i=11}^{15}\{-\frac{i+1}{i}+\frac{1}{i}\mathbb{Z}_{\geq 0}\}$\\ \hline
7& $\bigcup\limits_{i=8}^{16}\{-1+\frac{1}{i}\mathbb{Z}_{\geq 0}\}$\\ \hline
	
\end{longtable*}
\centering
\renewcommand{\arraystretch}{1.4}
\setlength\tabcolsep{5pt}
\caption{Diagonal-reducible points of $E_8$}
\begin{longtable*}{|c|c|}
	\hline
$p$ & diagonal-reducible points of $M_I(z\widehat{\xi_p})$  \\ 
	\endfirsthead
	\endfoot
	\endlastfoot
	\hline 
	1&$\bigcup\limits_{i=14}^{20}\{-1+\frac{1}{i}\mathbb{Z}_{\geq 0}\},\bigcup\limits_{i=21}^{27}\{-\frac{i+1}{i}+\frac{1}{i}\mathbb{Z}_{\geq 0}\}$\\ \hline
	2&$\bigcup\limits_{i=14}^{21}\{-\frac{i+1}{i}+\frac{1}{i}\mathbb{Z}_{\geq 0}\},\bigcup\limits_{i=22}^{26}\{-\frac{i+2}{i}+\frac{1}{i}\mathbb{Z}_{\geq 0}\}$\\ \hline
	3&$-\frac{14}{13}+\frac{1}{13}\mathbb{Z}_{\geq 0},\bigcup\limits_{i=14}^{20}\{-\frac{i+2}{i}+\frac{1}{i}\mathbb{Z}_{\geq 0}\},\bigcup\limits_{i=21}^{25}\{-\frac{i+3}{i}+\frac{1}{i}\mathbb{Z}_{\geq 0}\}$ \\ \hline
	4& $-\frac{7}{6}+\frac{1}{12}\mathbb{Z}_{\geq 0},\bigcup\limits_{i=13}^{15}\{-\frac{i+3}{i}+\frac{1}{i}\mathbb{Z}_{\geq 0}\},\bigcup\limits_{i=16}^{20}\{-\frac{i+4}{i}+\frac{1}{i}\mathbb{Z}_{\geq 0}\},\bigcup\limits_{i=21}^{23}\{-\frac{i+5}{i}+\frac{1}{i}\mathbb{Z}_{\geq 0}\}$\\ \hline
	5& $\bigcup\limits_{i=13}^{16}\{-\frac{i+2}{i}+\frac{1}{i}\mathbb{Z}_{\geq 0}\},\bigcup\limits_{i=17}^{21}\{-\frac{i+3}{i}+\frac{1}{i}\mathbb{Z}_{\geq 0}\},\bigcup\limits_{i=22}^{24}\{-\frac{i+4}{i}+\frac{1}{i}\mathbb{Z}_{\geq 0}\}$\\ \hline
	6&  $\bigcup\limits_{i=13}^{16}\{-\frac{i+1}{i}+\frac{1}{i}\mathbb{Z}_{\geq 0}\},\bigcup\limits_{i=17}^{23}\{-\frac{i+2}{i}+\frac{1}{i}\mathbb{Z}_{\geq 0}\},\bigcup\limits_{i=24}^{25}\{-\frac{i+3}{i}+\frac{1}{i}\mathbb{Z}_{\geq 0}\}$\\ \hline
	7& $\bigcup\limits_{i=13}^{26}\{-\frac{i+1}{i}+\frac{1}{i}\mathbb{Z}_{\geq 0}\}$\\ \hline
	8& $\bigcup\limits_{i=14}^{27}\{-1+\frac{1}{i}\mathbb{Z}_{\geq 0}\}$\\ \hline

\end{longtable*}

	

%
%
%
%

%
\bibliography{newestref}\bibliographystyle{alpha} 

@ARTICLE{BXX,
  author = {Bai, Z. Q. and Xiao, W. and Xie, X.},
  title = {Gelfand--{K}irillov dimensions and associated varieties of highest weight
	modules},
  journal = {Int. Math. Res. Not. IMRN},
  year = {2023},
 volume = {no. 10},
  pages = {8101-8142},

}

@Article{BX1,
  author    = {Bai, Z. Q. and Xie, X.},
  title     = {Gelfand--{K}irillov dimensions of highest weight {H}arish-{C}handra modules for {$SU(p,q)$}},
  journal   = {Int. Math. Res. Not. IMRN},
  year      = {2019},
 volume = {no. 14},
  pages     = {4392--4418},
}

@Article{BGXW,
  author    = {Bai, Z. Q. and Gao, F. and Wang, Y. T. and Xie, X.},
  title     = {Gelfand--{K}irillov dimensions and annihilator varieties of highest weight modules of exceptional type {L}ie algebras},
  year = {2025},
  journal   = {arXiv:2509.24346},
}

@BOOK{Ca85,
  title = {Finite groups of {L}ie type},
  publisher = {John Wiley \& Sons, Inc., New York},
  year = {1985},
  author = {Carter, R. W.},
  pages = {xii+544},
  series = {Pure and Applied Mathematics (New York)},
  note = {Conjugacy classes and complex characters, A Wiley-Interscience Publication},
  mrclass = {20G40 (20-02 20C15)},
  mrnumber = {794307},
  mrreviewer = {David B. Surowski}
}

@INCOLLECTION{EHW,
  author = {Enright, T. J. and Howe, R. and Wallach, N.},
  title = {A classification of unitary highest weight modules},
  booktitle = {Representation theory of reductive groups ({P}ark {C}ity, {U}tah,
	1982)},
  publisher = {Birkh\"auser Boston, Boston, MA},
  year = {1983},
  volume = {40},
  series = {Progr. Math.},
  pages = {97--143},
  mrclass = {22E46},
  mrnumber = {733809},
  mrreviewer = {M. Flensted-Jensen}
}

@article {HH,
	AUTHOR = {He, Haian},
	TITLE = {On the reducibility of scalar generalized {V}erma modules of
	abelian type},
	JOURNAL = {Algebr. Represent. Theory},
	FJOURNAL = {Algebras and Representation Theory},
	VOLUME = {19},
	YEAR = {2016},
	NUMBER = {1},
	PAGES = {147--170},
	ISSN = {1386-923X,1572-9079},
	MRCLASS = {17B20 (17B10)},
	MRNUMBER = {3465895},
	MRREVIEWER = {Eric\ N.\ Sommers},
	DOI = {10.1007/s10468-015-9567-2},
	URL = {https://doi.org/10.1007/s10468-015-9567-2},
}

@BOOK{Hum78,
  title = {Introduction to {L}ie algebras and representation theory},
  publisher = {Springer-Verlag, New York-Berlin},
  year = {1978},
  author = {Humphreys, J. E.},
  volume = {9},
  pages = {xii+171},
  series = {Graduate Texts in Mathematics},
  note = {Second printing, revised},
  mrclass = {17Bxx},
  mrnumber = {499562},
  mrreviewer = {I. P. Shestakov}
}

@Article{AJ1,
  author     = {Joseph, A.},
  journal    = {J. London Math. Soc. (2)},
  title      = {Gelfand--{K}irillov dimension for the annihilators of simple quotients of {V}erma modules},
  year       = {1978},
  number     = {1},
  pages      = {50--60},
  volume     = {18},
  fjournal   = {Journal of the London Mathematical Society. Second Series},
  mrclass    = {17B35},
  mrnumber   = {0506500},
  mrreviewer = {Marie-Paule Malliavin},
}

@ARTICLE{KL,
  author = {Kazhdan, D. and Lusztig, G.},
  title = {Representations of {C}oxeter groups and {H}ecke algebras},
  journal = {Invent. Math.},
  year = {1979},
  volume = {53},
  pages = {165--184},
  number = {2},
  fjournal = {Inventiones Mathematicae},
  mrclass = {20H15 (17B35 20G05 22E47)},
  mrnumber = {560412},
  mrreviewer = {Vinay V. Deodhar}
 
}

@BOOK{Knapp,
  title = {Lie groups beyond an introduction},
  publisher = {Birkh\"auser Boston, Inc., Boston, MA},
  year = {2002},
  author = {Knapp, A. W.},
  volume = {140},
  pages = {xviii+812},
  series = {Progress in Mathematics},
  edition = {Second},
  mrclass = {22-01},
  mrnumber = {1920389}
}

@BOOK{lusztig2003hecke,
  title = {Hecke algebras with unequal parameters},
  publisher = {American Mathematical Society, Providence, RI},
  year = {2003},
  author = {Lusztig, G.},
  volume = {18},
  pages = {vi+136},
  series = {CRM Monograph Series},
  mrclass = {20C08 (20F55)},
  mrnumber = {1974442 (2004k:20011)},
  mrreviewer = {G{\"o}tz Pfeiffer}
}

@INCOLLECTION{lusztig1985cellsI,
  author = {Lusztig, G.},
  title = {Cells in affine {W}eyl groups},
  booktitle = {Algebraic groups and related topics ({K}yoto/{N}agoya, 1983)},
  publisher = {North-Holland, Amsterdam},
  year = {1985},
  volume = {6},
  series = {Adv. Stud. Pure Math.},
  pages = {255--287},
  mrclass = {20G15},
  mrnumber = {803338 (87h:20074)},
  mrreviewer = {Bhama Srinivasan}
}

@Article{BJ,
  author  = {Bai, Z. Q. and J. Jiang},
  title   = {Gelfand--{K}irillov dimensions and Reducibility of scalar genralized Verma modules for classical {L}ie algebras},
  journal = {Acta Math. Sin. (Engl. Ser.)},
  year = {2024},
  volume = {40},
  pages = {658-706},
  number = {3},
}

@Article{BX,
  author  = {Bai, Z. Q. and W. Xiao},
  title   = {Gelfand--{K}irillov dimension and reducibility of scalar genralized {V}erma modules},
  journal = {Acta Math. Sin. (Engl. Ser.)},
  year = {2019},
  volume = {35},
  pages = {1854-1860},
  number = {11},
}

@Article{VD,
  author  = {D. N. Verma},
  title   = {Structure of certain induced representations of complex semisimple {L}ie algebras},
  journal = {Bull. Amer. Math. Soc.},
  volume  = {74},
  year    = {1968},
  pages   = {160-166},
}

@Article{JL,
  author  = {J. Lepowsky},
  title   = {A generalization of the {B}ernstein--{G}elfand--{G}elfand resolution},
  journal = {J. Algebra.},
  year    = {1977},
  volume  = {49},
  number  = {2},
  pages   = {496-511},
}

@Article{AR,
  author  = {A. Rocha-Caridi},
  title   = {Splitting criteria for {$\mathfrak{G}$}-modules induced from a parabolic and a {B}ernstein--{G}elfand--{G}elfand resolution of a finite-dimensional, irreducible {$\mathfrak{G}$}-module},
  journal = {Trans. Amer. Math. Soc.},
  year    = {1980},
  volume  = {262},
  number  = {2},
  pages   = {335-366},
}

@Article{BXiao,
  author  = {Z. Q. Bai and W. Xiao},
  title   = {Reducibility of generalized {V}erma modules for {H}ermitian symmetric pairs},
  journal = {J. Pure. Appl. Algebra.},
  year    = {2021},
  volume  = {225},
  number  = {4},
  pages   = {21},
}

@Article{HKZ,
  author  = {H. He and T. Kubo and R. Zierau},
  title   = {On the reducibility of scalar generalized {V}erma modules associated to maximal parabolic subalgebras},
  journal = {Kyoto J. Math.},
  year    = {2019},
  volume  = {59},
  number  = {4},
  pages   = {787-813},
}

@PhdThesis{Ku,
  author = {Kubo, T.},
  title  = {Conformally invariant systems of differential operators associated to two-step nilpotent maximal parabolics of non-{H}eisenberg type},
  school = {Oklahoma State University},
  year   = {2012},
}

@Book{MV,
  title      = {Generalized {V}erma {M}odules},
  publisher  = {VNTL Publishers, L'viv},
  year       = {2000},
  author     = {V. Mazorchuk},
  volume     = {8},
  series     = {Mathematical Studies Monograph Series},
  isbn       = {996-7148-90-4},
  mrclass    = {17B10},
  mrnumber   = {1844621},
  mrreviewer = {Brian D. Boe},
  pages      = {182},
}

@Article{JC,
  author  = {J. C. Jantzen},
  title   = {Contra-variant forms on induced representations of semi-simple {L}ie algebras},
  journal = {Math. Ann.},
  year    = {1977},
  volume  = {226},
  number  = {1},
  pages   = {53-65},
}

@Article{MH,
  author  = {H. Matumoto},
  title   = {The homomorphisms between scalar generalized {V}erma modules associated to maximal parabolic subalgebras},
  journal = {Duke Math. J.},
  year    = {2006},
  volume  = {131},
  number  = {1},
  pages   = {75-118},
}

@Article{AG1,
  author  = {A. Gyoja},
  title   = {Further generalization of {G}eneralized {V}erma modules},
  journal = {Publ. Res. Inst. Math. Sci.},
  year    = {1993},
  volume  = {29},
  number  = {3},
  pages   = {349-395},
}

@Article{JJ,
  author  = {J. Jiang},
  title   = {Reducibility of {S}calar  {G}eneralized {V}erma {M}odules of {M}inimal {P}arabolic {T}ype},
  journal = {Algebra Colloq.},
  year = {2025},
   volume  = {32},
  number  = {4},
  pages   = {623-634},
  doi = {10.1142/S1005386725000458},
  
}

@Article{Geck,
  author  = {M. Geck},
  title   = {PyCox:computing with (finite) Coxter groups and {I}wahori-{H}ecke algebras},
  journal = {J. Comput. Math.},
  year = {2012},
  volume = {15},
  pages = {231-256},
}

\end{document}